\documentclass{article}
\usepackage[left=4cm, right=4cm, top=4cm, bottom=2.25cm]{geometry}
\usepackage{amsmath,color}
\usepackage{amsfonts,amsthm}
\usepackage{amssymb,enumerate,enumitem,verbatim}
\usepackage[pdftex]{graphicx}
\usepackage{hyperref}
\usepackage{amsmath,setspace,scalefnt}
\usepackage{mathtools}
\usepackage[usenames,dvipsnames,svgnames,table]{xcolor}
\usepackage{amsfonts,amsthm}
\usepackage{amssymb,enumitem,verbatim}
\usepackage{graphicx}
\usepackage{tikz,caption,subcaption}
\usetikzlibrary{calc}
\usepackage[titles]{tocloft}
\usepackage[capitalise]{cleveref}
\usepackage{thmtools}
\usepackage{thm-restate}
\usepackage{dsfont}
\usepackage{appendix}

\newcounter{propcounter}
\makeatletter
\newcommand{\labelinthm}[1]{%
   \label{temp#1}
   \protected@write \@auxout {}{\string \newlabel{#1}{{\emph{\ref{temp#1}}}{\thepage}{\emph{\ref{temp#1}}}{temp#1}{}} }%
}
\makeatother

\global\long\def\N{\mathbb{N}}
\global\long\def\P{\mathbb{P}}
\global\long\def\E{\mathbb{E}}
\global\long\def\eps{\varepsilon}

\newtheorem{lemma}{Lemma}[section]
\newtheorem{corollary}[lemma]{Corollary}
\newtheorem{theorem}[lemma]{Theorem}

\newtheorem{conjecture}[lemma]{Conjecture}

\theoremstyle{definition}

\newtheorem{claim}[lemma]{Claim}

\newcommand{\claimproofstart}[1][Proof]{\begin{proof}[#1]
\renewcommand{\qedsymbol}{$\boxdot$}}
\newcommand\claimproofend{
\end{proof}
\renewcommand{\qedsymbol}{$\square$}}

\global\long\def\eps{\varepsilon}

\definecolor{darkgreen}{rgb}{0.1,0.7,0.1}

\newcommand{\hide}[1]{}

\title{Almost-full transversals in equi-$n$-squares}
\author{
Debsoumya Chakraborti\thanks{Mathematics Institute, University of Warwick, Coventry, United Kingdom.
E-mail: {\tt
\{debsoumya.chakraborti, richard.montgomery, teodor.petrov\}@warwick.ac.uk}.
 DC and RM supported by the European Research Council (ERC) under the European Union Horizon 2020 research and innovation programme (grant agreement No.\ 947978). TP  supported by the Warwick Mathematics Institute Centre for Doctoral Training, and by funding from the UK EPSRC (Grant number: EP/W524645/1).}
\and
Micha Christoph\thanks{Department of Computer Science, ETH Z\"{u}rich, Switzerland.
 E-mail: {\tt micha.christoph@inf.ethz.ch}.
 Supported by the SNSF Ambizione Grant No. 216071.
}
\and 
Zach Hunter\thanks{Department of Mathematics, ETH Zürich, Switzerland.
 E-mail: {\tt zach.hunter@ifor.math.ethz.ch}}.
\and
Richard Montgomery\footnotemark[1]
\and
Teo Petrov\footnotemark[1]
}

\begin{document}

\date{}

\maketitle

\begin{abstract} In 1975, Stein made a
wide generalisation of the Ryser-Brualdi-Stein conjecture on transversals in Latin squares, conjecturing that every equi-$n$-square (an $n\times n$ array filled with $n$ symbols where each symbol appears exactly $n$ times) has a transversal of size $n-1$. That is, it should have a collection of $n-1$ entries that share no row, column, or symbol.
In 2017, Aharoni, Berger, Kotlar, and Ziv showed that equi-$n$-squares always have a transversal with size at least $2n/3$.
 In 2019, Pokrovskiy and Sudakov disproved Stein's conjecture by constructing equi-$n$-squares without a transversal of size $n-\frac{\log n}{42}$, but asked whether Stein's conjecture is approximately true. I.e., does an equi-$n$-square always have a transversal with size $(1-o(1))n$?

We answer this question in the positive. More specifically, we improve both known bounds, showing that there exist equi-$n$-squares with no transversal of size $n-\Omega(\sqrt{n})$ and that every equi-$n$-square contains $n-n^{1-\Omega(1)}$ disjoint transversals of size $n-n^{1-\Omega(1)}$. Furthermore, we give results for the corresponding problem in globally-bounded colourings of complete graphs, as part of the wider framework of coloured complete directed graphs.
\end{abstract}

\section{Introduction}
A \emph{Latin square of order $n$} is an $n\times n$ array filled with $n$ symbols where each symbol appears exactly once in each row and each column. A \emph{transversal} (sometimes known as a \emph{partial transversal}) is a collection of cells in the array which share no row, column, or symbol, while a \emph{full transversal} in a Latin square of order $n$ is a transversal with $n$ cells. The \emph{size} of a transversal is its number of cells.
Latin squares, and transversals, have a long history of study dating back to Euler~\cite{OGeuler} who studied the decomposition of Latin squares into disjoint full transversals. As was known to Euler, a Latin square may have no full transversal. However, the prominent
Ryser-Brualdi-Stein conjecture~\cite{Ryser,Brualdi}, originating in the 1960's, suggests that every Latin square of order $n$ should have a transversal with $n-1$ cells and, moreover, one with $n$ cells if $n$ is odd. This conjecture has seen a lot of activity recently,  culminating in the proof by the fourth author~\cite{montgomery2023proof} that, when $n$ is sufficiently large, every Latin square of order $n$ has a transversal with $n-1$ cells. For more on this, related results, and the history of the study of Latin squares, see the recent surveys~\cite{Alexeysurvey,transversalsurvey}.

In 1975, Stein~\cite{stein1975transversals} made a series of bold conjectures that, broadly, suggest the Latin square conditions conjectured to guarantee a transversal of order $n-1$ may be overkill. In particular, he conjectured that any equi-$n$-square has a transversal of size $n-1$, where an equi-$n$-square is an $n\times n$ array filled with $n$ symbols which each appear exactly $n$ times. Thus, a Latin square of order $n$ is an equi-$n$-square where we additionally require every symbol to appear at most once in each row or column.
 As some evidence towards his conjecture, Stein~\cite{stein1975transversals} used the probabilistic method to show that any equi-$n$-square contains a transversal of size at least $(1-e^{-1})n$.
This bound was the state of the art for some 40 years, until Aharoni, Berger, Kotlar, and Ziv~\cite{aharoni2017conjecture} used topological methods to show that any equi-$n$-square contains a transversal of size at least $2n/3$. 
Very recently, Anastos and Morris~\cite{anastos2024notefindinglargetransversals} showed that any  equi-$n$-square contains a transversal of size at least $(3/4-o(1))n$. 
However, in 2019, Pokrovskiy and Sudakov~\cite{pokrovskiy2019counterexample} had shown that Stein's conjecture is, indeed, over-ambitious. That is, they constructed equi-$n$-squares that have no transversal with size larger than $n-\frac{1}{42}\log n$.

While this settles the falsity of Stein's conjecture, the bounds $n-O(\log n)$ and $(3/4-o(1))n$ on the size of the largest transversal that can be guaranteed in any equi-$n$-square are rather far apart. In particular, Pokrovskiy and Sudakov~\cite{pokrovskiy2019counterexample} (see also~\cite[Problem~4.3]{Alexeysurvey}) asked whether Stein's conjecture holds asymptotically, i.e., does an equi-$n$-square always have a transversal with $(1-o(1))n$ cells? Our main result is to confirm that it not only does, but that it can moreover be almost decomposed into such transversals, with the following result.

\begin{theorem}\label{thm:lowerbounddecomp}
There exists $\eps>0$ such that every equi-$n$-square contains at least $n-n^{1-\eps}$ disjoint transversals with size at least $n-n^{1-\eps}$.
\end{theorem}

Additionally, we show that a strengthening of Stein's conjecture, a conjecture made by Hahn~\cite{hahn1986path} concerning rainbow Hamilton paths in globally-bounded colourings of the complete graph, also holds true asymptotically. See the appendix for a discussion and a proof.

We also modify Pokrovskiy and Sudakov's construction from \cite{pokrovskiy2019counterexample} to show that a related but simpler construction provides equi-$n$-squares that must omit far more symbols in any transversal, as follows.

\begin{theorem}\label{thm:upperbound}
For each $n\in \mathbb{N}$, there is an equi-$n$-square with no transversal of size $n-\left(\frac{1}{2\sqrt{2}}+o(1)\right)\sqrt{n}$.
\end{theorem}

As our construction is very natural, it is reasonable to suggest that the constant $\frac{1}{2\sqrt{2}}$ in Theorem~\ref{thm:upperbound} is unimprovable. With slightly less ambition, we suggest the following  to replace Stein's disproved conjecture.

\begin{conjecture}
There exists a constant $C>0$ such that every equi-$n$-square has a transversal with size at least $n-C\sqrt{n}$.
\end{conjecture}

As is well-known, a Latin square $L$ of order $n$ has a natural representation as an $n$-regular linear 3-partite 3-uniform $(3n)$-vertex hypergraph $\mathcal{H}(L)$, where 3 vertex classes of size $n$ represent respectively the rows, columns, and symbols of $L$ and a 3-uniform edge of the row, column, and symbol is added to $\mathcal{H}(L)$ for each cell in $L$. As $L$ is an equi-$n$-square, $\mathcal{H}(L)$ is $n$-regular, and as $L$ is a Latin square each codegree in $\mathcal{H}(L)$ is at most 1 (i.e., $\mathcal{H}(L)$ is linear). From this perspective, the main question we have considered asks: in an $n$-regular 3-partite 3-uniform $(3n)$-vertex hypergraph $\mathcal{H}$ with vertex classes $A,B$ and $C$ of size $n$, if the codegree of $x$ and $y$ is at most 1 for each $x\in A$ and $y\in B$, then does $\mathcal{H}$ have a transversal of size $(1-o(1))n$ (or, indeed, an almost decomposition into such transversals)? Thus, for Theorem~\ref{thm:lowerbounddecomp}, we dispense with $2/3$rds of the codegree conditions for the Latin square case. Note that we cannot dispense with all of the codegree conditions here, due to the following construction adapted from an example of Alon and Kim~\cite{alon1997degree}.

Let $t\in \N$, $X=\{x_1,\ldots,x_{2t},x_1',\ldots,x_t'\}$, $Y=\{y_1,\ldots,y_{2t},y_1',\ldots,y_t'\}$ and $Z=\{z_1,\ldots,z_{2t},z_1',\ldots,z_t'\}$. Let $V(\mathcal{H})=X\cup Y\cup Z$ and
\[
E(\mathcal{H})=\{(x_i,y_i,z_j'),(x_i,y'_j,z_i),(x'_j,y_i,z_i):i\in [2t],j\in [t]\}.
\]
Note that any matching in $\mathcal{H}$ can contain at most 1 edge intersecting $\{x_i,y_i,z_i\}$ for each $i\in [2t]$, and therefore any matching in $\mathcal{H}$ can have at most $2t$ edges, and hence cover at most $6t=2|\mathcal{H}|/3$ vertices. Moreover, $\mathcal{H}$ is $2t$-regular and 3-partite. As $\mathcal{H}$ has $n:=3t$ vertices in each class, this does not quite correspond to our case, but blowing up each vertex in $\mathcal{H}$ by 3 gives an $18t$-regular 3-partite 3-graph in which any matching must omit at least $1/9$ of the vertices and in which there are $3n=9t$ vertices in each class. Then, taking 2 disjoint copies of this graph gives an $18t$-regular 3-partite 3-graph with $n'=3\cdot 18t$ vertices in which any perfect matching must omit at least $n'/9$ vertices.

Later, in Section~\ref{sec:deduce}, we will deduce Theorem~\ref{thm:lowerbounddecomp} from a more general result, Theorem~\ref{thm:semiunifieddecomp}. Similarly, we could deduce a version of Theorem~\ref{thm:lowerbounddecomp} with further weakened codegree conditions to say that (using the notation above) $\mathrm{cod}_{\mathcal{H}}(x,y)\leq n^{1-\mu}$, as long as $\eps\ll \mu$. Here, the corresponding square has $n$ copies of each symbol, and $n$ symbols in each row and column, but we allow up to $n^{1-\mu}$ symbols to appear in any one square as long as they are all different.

In Section~\ref{sec:upperbound}, we prove Theorem~\ref{thm:upperbound}. In Section~\ref{sec:simpler}, we sketch a simplified version of our methods for Theorem~\ref{thm:lowerbounddecomp}, showing how to find a large transversal in certain equi-$n$-squares. In Section~\ref{sec:lowerbound}, we then prove Theorem~\ref{thm:lowerbounddecomp}. A proof of Theorem~\ref{thm:cyclictransversalinarray} is provided in the appendix.

\section{Equi-{\boldmath $n$}-squares with no partial transversal of size {\boldmath $n-\Omega(\sqrt{n})$}}\label{sec:upperbound}
To illustrate the main ideas of our construction, we first prove \Cref{thm:upperbound} whenever $n=2m^2$ for some integer $m$, constructing equi-$n$-squares with no transversal of size larger than $n-\frac{1}{2\sqrt{2}}\sqrt{n}$ as depicted in Figure~\ref{fig:simplifiedcase}a). We will then show how this construction can be modified for general $n$ without changing the size of the largest possible transversal by much. In our construction we use a key part of the construction by Pokrovskiy and Sudakov~\cite{pokrovskiy2019counterexample} of equi-$n$-squares with no transversal of size more than $n-\frac{1}{42}\log n$, which, roughly, is the argument depicted in Figure~\ref{fig:simplifiedcase}b).

\begin{proof}[\bf{\boldmath{Proof of \Cref{thm:upperbound} whenever $n = 2m^2$ for some integer $m$}}]
We construct an equi-$n$-square $S$ with both rows and columns indexed by elements in $[n]$, where $n=2m^2$ for some $m\in \mathbb{Z}$. We use the notation $(x,y)$ to denote the cell in the $x$-th row and $y$-th column. We partition $S$ into $(2m)^2$ \textit{boxes}, each containing $m^2$ cells, as follows. For every $i,j \in [2m]$, let
\[
A_{i,j}=\{ (x,y) : 1+(i-1)m \leq x \leq im ~~ \text{and} ~~1+(j-1)m \leq y \leq jm \}.
\]
We now pair up the boxes and assign a unique colour to each pair, noting that this uses $\frac{(2m)^2}{2}=2m^2=n$ colours, as depicted in Figure~\ref{fig:simplifiedcase}a). For each $k\in [m]$, we pair the boxes $A_{2k-1,2k-1}$ and $A_{2k,2k}$ and colour every cell they contain the same colour. For each $1\leq i<j\leq 2m$, we pair the boxes $A_{i,j}$ and $A_{j,i}$ and colour every cell they contain the same colour.
Notice that as each colour appears in exactly $2m^2=n$ many cells, we have constructed an equi-$n$-square.

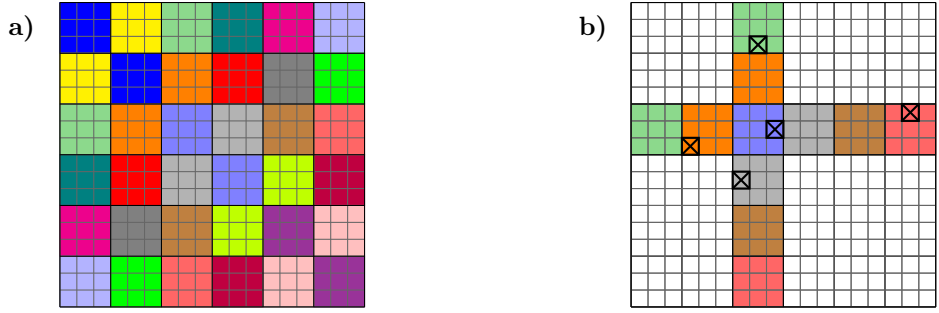
\begin{figure}[t]
\begin{center}
\begin{tikzpicture}[define rgb/.code={\definecolor{mycolor}{rgb}{#1}}, rgb color/.style={define rgb={#1},mycolor},scale=0.9]

\def\wi{0.75cm}
\def\n{5}

\foreach \x in {0,1,...,\n,6}
\foreach \y in {0,1,...,\n,6}
{
\coordinate (A\x\y) at ($\n*(0,\wi)+\x*(\wi,0)-\y*(0,\wi)$);
}

\foreach \x/\y/\col in {0/0/blue,1/1/blue,2/2/blue!50,3/3/blue!50,4/4/violet!80,5/5/violet!80}
{
\draw (A\x\y) [fill=\col,\col] rectangle ($(A\x\y)+(\wi,\wi)$);
}

\foreach \x/\y/\col in {0/1/yellow,0/2/darkgreen!50,0/3/teal,2/1/orange,0/4/magenta,0/5/blue!30,3/5/purple,2/5/red!60,1/5/green,4/5/pink,3/4/lime,2/4/brown,1/4/gray,1/3/red,2/3/black!30}
{
\draw (A\x\y) [thin,\col,fill=\col] rectangle ($(A\x\y)+(\wi,\wi)$);
\draw (A\y\x) [thin,\col,fill=\col] rectangle ($(A\y\x)+(\wi,\wi)$);
}
\draw ($(A00)-(0.75*\wi,0)+0.5*(0,\wi)$) node {\textbf{a)}};

\foreach \x in {0,1,...,\n}
{
\foreach \z in {0,1,...,2}
{
\draw [black!60] ($(A\x\n)+0.33333333*(\z*\wi,0)$) --  ($(A\x\n)+0.33333333*(\z*\wi,0)+\n*(0,\wi)+(0,\wi)$);
\draw [black!60] ($(A0\x)+0.33333333*(0,\z*\wi)$) --  ($(A0\x)+0.33333333*(0,\z*\wi)+\n*(\wi,0)+(\wi,0)$);
}
}

\foreach \x in {0,1,...,\n}
{
{
\draw [black] ($(A\x\n)$) --  ($(A\x\n)+\n*(0,\wi)+(0,\wi)$);
\draw [black] ($(A0\x)$) --  ($(A0\x)+\n*(\wi,0)+(\wi,0)$);
}
}
\draw [black] ($(A55)+(\wi,0)$) --  ($(A55)+\n*(0,\wi)+(\wi,\wi)$);
\draw [black] ($(A00)+(0,\wi)$) --  ($(A00)+\n*(\wi,0)+(\wi,\wi)$);
\end{tikzpicture}
\;\;\;\;\;\hspace{2cm}
\begin{tikzpicture}[define rgb/.code={\definecolor{mycolor}{rgb}{#1}}, rgb color/.style={define rgb={#1},mycolor},scale=0.9]
\def\wi{0.75cm}
\def\n{5}

\foreach \x in {0,1,...,\n,6}
\foreach \y in {0,1,...,\n,6}
{
\coordinate (A\x\y) at ($\n*(0,\wi)+\x*(\wi,0)-\y*(0,\wi)$);
}

\foreach \x/\y/\col in {2/2/blue!50}
{
\draw (A\x\y) [fill=\col,\col] rectangle ($(A\x\y)+(\wi,\wi)$);
}

\foreach \x/\y/\col in {0/2/darkgreen!50,2/1/orange,2/5/red!60,2/4/brown,2/3/black!30}
{
\draw (A\x\y) [thin,\col,fill=\col] rectangle ($(A\x\y)+(\wi,\wi)$);
\draw (A\y\x) [thin,\col,fill=\col] rectangle ($(A\y\x)+(\wi,\wi)$);
}

\draw ($(A00)-(0.75*\wi,0)+0.5*(0,\wi)$) node {\textbf{b)}};

\foreach \x in {0,1,...,\n}
{
\foreach \z in {0,1,...,2}
{
\draw [black!60] ($(A\x\n)+0.333333333333333333333*(\z*\wi,0)$) --  ($(A\x\n)+0.333333333333333333333*(\z*\wi,0)+\n*(0,\wi)+(0,\wi)$);
\draw [black!60] ($(A0\x)+0.333333333333333333333*(0,\z*\wi)$) --  ($(A0\x)+0.333333333333333333333*(0,\z*\wi)+\n*(\wi,0)+(\wi,0)$);
}
}

\foreach \x in {0,1,...,\n}
{
\draw [black] ($(A\x\n)$) --  ($(A\x\n)+\n*(0,\wi)+(0,\wi)$);
\draw [black] ($(A0\x)$) --  ($(A0\x)+\n*(\wi,0)+(\wi,0)$);
}
\draw [black] ($(A55)+(\wi,0)$) --  ($(A55)+\n*(0,\wi)+(\wi,\wi)$);
\draw [black] ($(A00)+(0,\wi)$) --  ($(A00)+\n*(\wi,0)+(\wi,\wi)$);
\foreach \x in {0,1,...,\n}
{
\foreach \z in {0,1,...,2}
{
\draw [black!60] ($(A\x\n)+0.333333333333333333333*(\z*\wi,0)$) --  ($(A\x\n)+0.333333333333333333333*(\z*\wi,0)+\n*(0,\wi)+(0,\wi)$);
\draw [black!60] ($(A0\x)+0.333333333333333333333*(0,\z*\wi)$) --  ($(A0\x)+0.333333333333333333333*(0,\z*\wi)+\n*(\wi,0)+(\wi,0)$);
}
}
\foreach \x in {0,2,3}
{
\draw [black] ($(A\x\n)$) --  ($(A\x\n)+\n*(0,\wi)+(0,\wi)$);
}
\foreach \x in {1,2,5}
{
\draw [black] ($(A0\x)$) --  ($(A0\x)+\n*(\wi,0)+(\wi,0)$);
}
\draw [black] ($(A55)+(\wi,0)$) --  ($(A55)+\n*(0,\wi)+(\wi,\wi)$);
\draw [black] ($(A00)+(0,\wi)$) --  ($(A00)+\n*(\wi,0)+(\wi,\wi)$);

\foreach \x/\y/\z/\zz in {2/2/2/2,2/3/0/2,2/0/1/1,1/2/0/1,5/2/1/3}
{
\draw [black,thick] ($(A\x\y)+0.333333333333333333333*(\z*\wi,\zz*\wi)$) rectangle ($(A\x\y)+0.333333333333333333333*(\z*\wi,\zz*\wi)+0.33333333*(\wi,-\wi)$);
\draw [black,thick] ($(A\x\y)+0.333333333333333333333*(\z*\wi,\zz*\wi)$) -- ++($0.333333333333333333333*(\wi,-\wi)$);
\draw [black,thick] ($(A\x\y)+0.333333333333333333333*(\z*\wi,\zz*\wi)+0.333333333333333333333*(\wi,0)$) -- ++($0.333333333333333333333*(-\wi,-\wi)$);
}

\foreach \x/\y in {2/4,2/3,2/0}
{
\draw ($(A\x\y)$) -- ++(\wi,0);
}
\foreach \x/\y in {4/2,5/2,1/2}
{
\draw ($(A\x\y)$) -- ++(0,\wi);
}

\end{tikzpicture}
\end{center}
\caption{\textbf{a)} Our construction in the simplified case when $n=2m^2$ with $m=3$. \textbf{b)} The boxes comprising $C_3$ in the same equi-$n$-square with a transversal highlighted -- if a light blue square appears in $C_3$ in a transversal then at most 4 of the 5 colours appearing only in $C_3$ can appear. Thus, at least one colour from $C_3\cup C_4$ is omitted in any transversal.\label{fig:simplifiedcase}}
\end{figure}

Letting $S$ be the equi-$n$-square we have constructed, we now show it has no transversal with more than $n-m/2$ cells. Suppose, then, that $T$ is a transversal of $S$. For each $k \in [2m]$, let
\[
C_k=\Big(\bigcup_{i\in [2m]} A_{i,k}\Big) \bigcup \Big(\bigcup_{j\in [2m]} A_{k,j}\Big),
\]
noting this is the disjoint union of $4m-1$ boxes. As depicted in Figure~\ref{fig:simplifiedcase}b), we now show the following claim.

\begin{claim}\label{clm:missing one colour}
    For every $k\in [m]$, at least one colour used in $C_{2k-1}\cup C_{2k}$ is not used in $T$.
\end{claim}
\begin{proof}[Proof of Claim~\ref{clm:missing one colour}] Let $k\in [m]$, and suppose that the colour in boxes $A_{2k-1,2k-1}$ and $A_{2k,2k}$ is blue.
If $T$ misses out on the colour blue, we would be done. Suppose then that $T$ contains a blue cell in $A_{j,j}$ for $j\in \{2k-1,2k\}$.
As $\bigcup_{i\in [2m]} A_{i,j}\subset C_j$ contains elements only in $m$ columns of $S$, $T$ has at most $m$ cells in $\bigcup_{i\in [2m]} A_{i,j}$. Similarly, $T$ has at most $m$ cells in $\bigcup_{i\in [2m]} A_{j,i}\subset C_j$.
As $T$ has a cell in $A_{j,j}$, it therefore has at most $2m-1$ cells in $C_j$, and thus at most $2m-2$ cells in $C_j$ which are not blue. However, $C_j$ uses $2m-1$ non-blue colours, each of which do not appear in $S \setminus C_j$. Thus, $T$ misses out on at least one colour used on $C_j\subset C_{2k-1}\cup C_{2k}$.
\renewcommand{\qedsymbol}{$\boxdot$}
\end{proof}
\renewcommand{\qedsymbol}{$\square$}
\noindent To finish, observe that each colour appears on at most two of $C_{2k-1}\cup C_{2k}$, $k \in [m]$. This combined with \Cref{clm:missing one colour} implies that $T$ must miss out on at least $\frac{m}{2}$ colours, and therefore $|T| \leq n-\frac{m}{2}=n -\frac{\sqrt{n}}{2\sqrt{2}}$. 
\end{proof}

We now deal with the general $n$ case.
Above we considered boxes of equal side lengths and now we will allow the side lengths to differ. We will choose these lengths so that they are close to each other and that the number of cells in each box is as close to $n/2$ as possible, though in general slightly below this. We use a similar construction to the special case, but each colour may need to be used in $O(n^{1/4})$ cells outside of its pair of boxes to make up $n$ cells of that colour. These $O(n\cdot n^{1/4})$ cells will be grouped together into $O(n^{1/4})$ rows and columns, so that any transversal would contain at most $O(n^{1/4})$ of them, which will not affect significantly the argument given above for the maximum size of a transversal.

\begin{proof}[\bf{Proof of \Cref{thm:upperbound}}]
Once again, we will index the rows and columns of the equi-$n$-square by elements of $[n]$. We will use rectangular boxes and so now choose the side lengths, denoted $a$ and $b$. 
For this let $m = \lceil \sqrt{n/2}\rceil$, $r = \lceil \sqrt{m^2-n/2}\rceil$ and take $a=m+r,b=m-r$. We first note that
$0\le m^2-n/2\le 2m\le 4\sqrt{n/2}$, and then, similarly, $0\le r^2-(m^2-n/2)\le 2r\le 4\sqrt{4\sqrt{n/2}}\le 8n^{1/4}$. Consequently, as  $n-2ab= 2(n/2-(m^2-r^2))$,
\[n-16n^{1/4}\le 2ab\le n.\]
Moreover, $b\ge \sqrt{n/2}-r>\sqrt{n/2}-8n^{1/4}$.

For each $i\in [2a]$ and $j\in [2b]$ take the box
\[
B_{i,j}=\{ (x,y) : 1+(i-1)b \leq x \leq ib ~~ \text{and} ~~ 1+(j-1)a \leq y \leq ja \}.
\]
As before, we pair up the boxes and assign a unique colour to each pair, using that there will be $2ab\leq n$ pairs of boxes, so that each pair can have its own colour. Note that we must use $n - 2ab \le 16n^{1/4}$ entries of each colour elsewhere in the equi-$n$-square. Noting also that $b\leq a$, we assign the same colour to $B_{2k-1,2k-1}$ and $B_{2k,2k}$ for each $k \in [b]$.
We assign the same colour to $B_{i,j}$ and $B_{j,i}$ for each $1\leq i<j\leq b$.
Additionally, for each $b+1\leq s \leq a$ and $t\in [2b]$, we assign the same colour to $B_{2s-1,t}$ and $B_{2s,t}$. Finally, we use the leftover of the colours used as well as the other colours we have not yet used to arbitrarily colour the cells which have not yet been coloured. This finishes the construction of our equi-$n$-square.

Letting $S$ be the constructed square, we now show it has no transversal with size $n-\left(\frac{1}{2\sqrt{2}}+o(1)\right)\sqrt{n}$. Suppose then that $T$ is a transversal of $S$.
Let $S'$ be the sub-square consisting of all the cells in the first $2ab$ rows and $2ab$ columns of $S$, and let $T'=T \cap S'$. Note that $|T \cap (S\setminus S')| \leq 2(n-2ab)\leq 32n^{1/4}$. We now solely focus on $T'$ and use an almost identical argument to the one in the special case done before.
For each $k \in [2b]$, let
\[
C_k:=\Big(\bigcup_{i\in [2a]} B_{i,k}\Big) \bigcup \Big(\bigcup_{j\in [2b]} B_{k,j}\Big).
\]
Note that every $C_k$ is the disjoint union of $2a+2b-1$ boxes. We now show the corresponding version of Claim~\ref{clm:missing one colour}.

\begin{claim}\label{clm:general case missing one colour}
    For every $k\in [b]$, at least one colour used in $C_{2k-1}\cup C_{2k}$ is not used in $T'$.
\end{claim}
\begin{proof}[Proof of Claim~\ref{clm:general case missing one colour}]
Suppose that the colour in boxes $B_{2k-1,2k-1}$ and $B_{2k,2k}$ is blue.
If $T'$ has no blue cell, we would be done. Hence we assume that for some $j\in \{2k-1,2k\}$, $T'$ contains a cell in $B_{j,j}$. As $\bigcup_{i\in [2a]} B_{i,j}\subset C_j$ contains elements in exactly $a$ columns of $S'$, $T'$ has at most $a$ cells in $\bigcup_{i\in [2a]} B_{i,j}$. Similarly, $T'$ has at most  $b$ cells in $\bigcup_{i\in [2b]} B_{j,i}$. As $T'$ has a cell in $B_{j,j}$, it therefore has at most $a+b-1$ cells in $C_j$, and thus at most $a+b-2$ cells in $C_j$ which are not blue. However, $C_j$ uses $(1/2)(2a-1+2b-1)= a+b-1$ non-blue colours, each of which does not appear in $S' \setminus C_j$. Thus, $T'$ misses out on at least one colour used on $C_j \subset C_{2k-1}\cup C_{2k}$.
\renewcommand{\qedsymbol}{$\boxdot$}
\end{proof}
\renewcommand{\qedsymbol}{$\square$}
Observe that, within $S'$, any colour appears in at most two of the sets $C_{2k-1}\cup C_{2k}$ for different $k \in [b]$. This combined with \Cref{clm:general case missing one colour} implies that $T'$ must miss out on at least $b/2$ colours. Thus,
\[
|T|=|T'|+|T \cap (S\setminus S')|\leq n-\frac{b}{2}+32n^{1/4}{\leq} n-\frac{\sqrt{n}}{2\sqrt{2}}+36n^{1/4}=n-(1-o(1))\frac{\sqrt{n}}{2\sqrt{2}}.\qedhere
\]
\end{proof}


\section{{\boldmath Proof sketch: one large transversal in certain equi-$n$-squares}}\label{sec:simpler}

Suppose $S$ is an equi-$n$-square. When each symbol appears $o(n)$ times in each row and column, a large transversal is known to exist via the R\"odl nibble, where this large transversal satisfies certain pseudorandom conditions. Thus, it is important to consider equi-$n$-squares in which symbols appear many times in the same row or the same column -- to discuss our methods we will choose a special case where each symbol is repeated many times in a small number of columns. Suppose that $n=4m$ with $m\in \N$, and that the cells of $S$ can be partitioned into a set of blocks $\mathcal{B}$, where each $B\in \mathcal{B}$ contains $m$ cells which are all in the same column and all have the same symbol (see, for example, Figure~\ref{fig:simplifiedexample}). Using Hall's matching criterion it is not hard to show that the columns of $S$ can be perfectly matched into $\mathcal{B}$ with a matching $M$ such that  each column is matched to a block in that column, and every block in this matching has a different symbol. Let $S'$ be the subsquare of $S$ formed by retaining only those cells whose block appears in an edge of $M$, and note that there are $m$ entries in each column of $S'$.
Our aim (at the expense of finding only an \emph{almost}-perfect matching) is to find a random such matching $M$ for which it is likely that almost every row has around $m$ entries in $S'$. Using (a defect version of) Hall's matching criterion in this roughly regular $S'$, it is not hard then to find a collection of $(1-o(1))n$ cells in $S'$ which share no row or column, whereupon the definition of $S'$ implies they share no symbol. Thus, we will have found a large transversal in $S$ (see, again, Figure~\ref{fig:simplifiedexample}). We will find our random matching $M$ using a novel `bounded-dependence matching algorithm'.

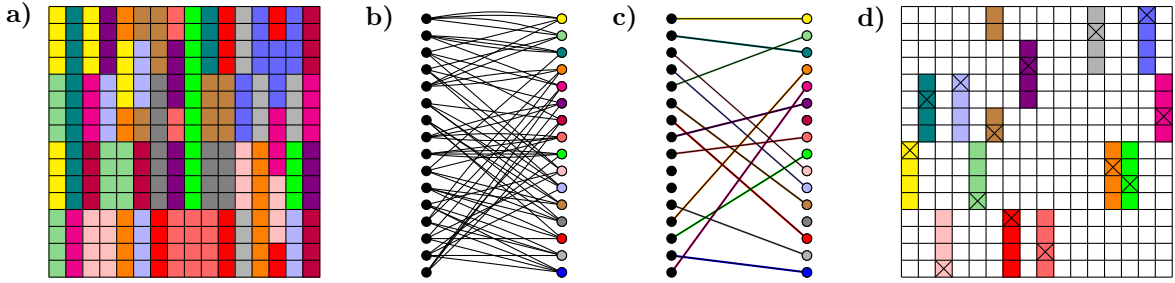
\begin{figure}[t]
\begin{center}
\begin{tikzpicture}[define rgb/.code={\definecolor{mycolor}{rgb}{#1}}, rgb color/.style={define rgb={#1},mycolor},scale=0.9]

\def\wi{0.25cm}
\def\n{16}

\foreach \x in {0,1,...,\n}
\foreach \y in {0,1,...,\n}
{
\coordinate (A\x-\y) at ($\n*(0,\wi)+\x*(\wi,0)-\y*(0,\wi)$);
}


\foreach \x/\y/\col in {0/0/yellow,1/0/teal,2/0/yellow,3/0/violet,4/0/orange,5/0/brown,6/0/brown,7/0/red!60,8/0/green,9/0/teal,10/0/red,11/0/black!30,12/0/blue!60,13/0/red,14/0/blue!60,15/0/purple,0/1/yellow,1/1/teal,2/1/yellow,3/1/violet,4/1/orange,5/1/brown,6/1/brown,7/1/red!60,8/1/green,9/1/teal,10/1/red,11/1/black!30,12/1/blue!60,13/1/red,14/1/blue!60,15/1/purple,0/2/yellow,1/2/teal,2/2/yellow,3/2/violet,4/2/yellow,5/2/blue!30,6/2/brown,7/2/violet,8/2/green,9/2/teal,10/2/red,11/2/black!30,12/2/blue!60,13/2/blue!60,14/2/blue!60,15/2/purple,0/3/yellow,1/3/teal,2/3/yellow,3/3/violet,4/3/yellow,5/3/blue!30,6/3/brown,7/3/violet,8/3/green,9/3/teal,10/3/red,11/3/black!30,12/3/blue!60,13/3/blue!60,14/3/blue!60,15/3/purple,0/4/darkgreen!50,1/4/teal,2/4/magenta,3/4/blue!30,4/4/yellow,5/4/blue!30,6/4/gray,7/4/violet,8/4/green,9/4/brown,10/4/brown,11/4/blue!60,12/4/black!30,13/4/blue!60,14/4/black!30,15/4/magenta,0/5/darkgreen!50,1/5/teal,2/5/magenta,3/5/blue!30,4/5/yellow,5/5/blue!30,6/5/gray,7/5/violet,8/5/green,9/5/brown,10/5/brown,11/5/blue!60,12/5/black!30,13/5/blue!60,14/5/black!30,15/5/magenta,0/6/darkgreen!50,1/6/teal,2/6/magenta,3/6/blue!30,4/6/orange,5/6/brown,6/6/gray,7/6/red!60,8/6/green,9/6/brown,10/6/brown,11/6/blue!60,12/6/black!30,13/6/magenta,14/6/black!30,15/6/magenta,0/7/darkgreen!50,1/7/teal,2/7/magenta,3/7/blue!30,4/7/orange,5/7/brown,6/7/gray,7/7/red!60,8/7/green,9/7/brown,10/7/brown,11/7/blue!60,12/7/black!30,13/7/magenta,14/7/black!30,15/7/magenta,0/8/yellow,1/8/teal,2/8/purple,3/8/darkgreen!50,4/8/darkgreen!50,5/8/purple,6/8/gray,7/8/violet,8/8/green,9/8/gray,10/8/gray,11/8/pink,12/8/orange,13/8/magenta,14/8/green,15/8/violet,0/9/yellow,1/9/teal,2/9/purple,3/9/darkgreen!50,4/9/darkgreen!50,5/9/purple,6/9/gray,7/9/violet,8/9/green,9/9/gray,10/9/gray,11/9/pink,12/9/orange,13/9/magenta,14/9/green,15/9/violet,0/10/yellow,1/10/teal,2/10/purple,3/10/darkgreen!50,4/10/darkgreen!50,5/10/purple,6/10/gray,7/10/violet,8/10/green,9/10/gray,10/10/gray,11/10/pink,12/10/orange,13/10/pink,14/10/green,15/10/violet,0/11/yellow,1/11/teal,2/11/purple,3/11/darkgreen!50,4/11/darkgreen!50,5/11/purple,6/11/gray,7/11/violet,8/11/green,9/11/gray,10/11/gray,11/11/pink,12/11/orange,13/11/pink,14/11/green,15/11/violet,0/12/darkgreen!50,1/12/magenta,2/12/pink,3/12/pink,4/12/orange,5/12/blue!30,6/12/red,7/12/red!60,8/12/red!60,9/12/red!60,10/12/red,11/12/black!30,12/12/orange,13/12/pink,14/12/blue!30,15/12/purple,0/13/darkgreen!50,1/13/magenta,2/13/pink,3/13/pink,4/13/orange,5/13/blue!30,6/13/red,7/13/red!60,8/13/red!60,9/13/red!60,10/13/red,11/13/black!30,12/13/orange,13/13/pink,14/13/blue!30,15/13/purple,0/14/darkgreen!50,1/14/magenta,2/14/pink,3/14/pink,4/14/orange,5/14/blue!30,6/14/red,7/14/red!60,8/14/red!60,9/14/red!60,10/14/red,11/14/black!30,12/14/orange,13/14/red,14/14/blue!30,15/14/purple,0/15/darkgreen!50,1/15/magenta,2/15/pink,3/15/pink,4/15/orange,5/15/blue!30,6/15/red,7/15/red!60,8/15/red!60,9/15/red!60,10/15/red,11/15/black!30,12/15/orange,13/15/red,14/15/blue!30,15/15/purple}
{
\draw (A\x-\y) [thin,\col,fill=\col] rectangle ($(A\x-\y)+(\wi,\wi)$);
}

\draw ($(A0-0)-(1.75*\wi,0)+0.5*(0,\wi)$) node {\textbf{a)}};

\foreach \x in {0,1,...,\n}
{
{
\draw [black] ($(A\x-\n)+(0,\wi)$) --  ($(A\x-\n)+\n*(0,\wi)+(0,\wi)$);
\draw [black] ($(A0-\x)+(0,\wi)$) --  ($(A0-\x)+\n*(\wi,0)+(0,\wi)$);
}
}

\end{tikzpicture}
\hspace{0.25cm}
\begin{tikzpicture}[define rgb/.code={\definecolor{mycolor}{rgb}{#1}}, rgb color/.style={define rgb={#1},mycolor},scale=0.9]

\def\wi{0.25cm}
\def\spacer{2cm}
\def\n{16}
\def\rad{0.07cm}

\foreach \x in {0,1}
\foreach \y in {1,...,\n}
{
\coordinate (A\x-\y) at ($\n*(0,\wi)+\x*(\spacer,0)-\y*(0,\wi)+(0,\wi)$);
}

\draw (A0-3) -- (A1-1);		\draw (A0-5) -- (A1-1);
  \draw (A0-4) -- (A1-2);	\draw (A0-5) -- (A1-2);
              \draw (A0-10) -- (A1-3);

\draw (A0-2) -- (A1-5);	\draw (A0-3) -- (A1-5);											\draw (A0-14) -- (A1-5);		\draw (A0-16) -- (A1-5);
  \draw (A0-4) -- (A1-6);												\draw (A0-16) -- (A1-6);
\draw (A0-3) -- (A1-7);			\draw (A0-6) -- (A1-7);
            \draw (A0-9) -- (A1-8);	\draw (A0-10) -- (A1-8);
                        \draw (A0-15) -- (A1-9);
\draw (A0-3) -- (A1-10);	\draw (A0-4) -- (A1-10);								\draw (A0-12) -- (A1-10);		\draw (A0-14) -- (A1-10);
  \draw (A0-4) -- (A1-11);											\draw (A0-15) -- (A1-11);
      \draw (A0-6) -- (A1-12);	\draw (A0-7) -- (A1-12);			\draw (A0-10) -- (A1-12);	\draw (A0-11) -- (A1-12);
              \draw (A0-10) -- (A1-13);	\draw (A0-11) -- (A1-13);
        \draw (A0-7) -- (A1-14);							\draw (A0-14) -- (A1-14);
                    \draw (A0-13) -- (A1-15);		\draw (A0-15) -- (A1-15);
                  \draw (A0-12) -- (A1-16);	\draw (A0-13) -- (A1-16);	\draw (A0-14) -- (A1-16);	\draw (A0-15) -- (A1-16);

\draw (A0-1)to[out=10, in=170, relative=true] (A1-1);
\draw (A0-1)to[out=5, in=175,relative=true] (A1-1);
\draw (A0-1) to[out=10, in=170, relative=true] (A1-2);
\draw (A0-1) -- (A1-2);

\draw (A0-5)to[out=5, in=175,relative=true] (A1-4);
\draw (A0-5)to[out=-5, in=-175,relative=true] (A1-4);

\draw (A0-13) to[out=5, in=175,relative=true] (A1-4);
\draw (A0-13) to[out=-5, in=-175,relative=true] (A1-4);

\draw (A0-8) to[out=5, in=175,relative=true] (A1-6);
\draw (A0-8) to[out=-5, in=-175,relative=true] (A1-6);

\draw (A0-16) to[out=5, in=175,relative=true] (A1-7);
\draw (A0-16) to[out=-5, in=-175,relative=true] (A1-7);

\draw (A0-8) to[out=5, in=175,relative=true] (A1-8);
\draw (A0-8) to[out=-5, in=-175,relative=true] (A1-8);

\draw (A0-6) to[out=5, in=175,relative=true] (A1-11);
\draw (A0-6) to[out=-5, in=-175,relative=true] (A1-11);

\draw (A0-7) to[out=5, in=175,relative=true] (A1-13);
\draw (A0-7) to[out=-5, in=-175,relative=true] (A1-13);

\draw (A0-11) to[out=5, in=175,relative=true] (A1-14);
\draw (A0-11) to[out=-5, in=-175,relative=true] (A1-14);

\draw (A0-12) to[out=5, in=175,relative=true] (A1-15);
\draw (A0-12) to[out=-5, in=-175,relative=true] (A1-15);


\draw (A0-2) -- (A1-3);
\draw (A0-2) to[out=5, in=175,relative=true] (A1-3);
\draw (A0-2) to[out=-5, in=-175,relative=true] (A1-3);

\draw (A0-9) -- (A1-9);
\draw (A0-9) to[out=5, in=175,relative=true] (A1-9);
\draw (A0-9) to[out=-5, in=-175,relative=true] (A1-9);

\foreach \x in {0,1}
\foreach \y in {1,...,\n}
{
\draw (A\x-\y) [fill] circle [radius=\rad];
}

\draw (A1-1) [fill=yellow] circle [radius=\rad];
\draw (A1-2) [fill=darkgreen!50] circle [radius=\rad];
\draw (A1-3) [fill=teal] circle [radius=\rad];
\draw (A1-4) [fill=orange] circle [radius=\rad];
\draw (A1-5) [fill=magenta] circle [radius=\rad];
\draw (A1-6) [fill=violet] circle [radius=\rad];
\draw (A1-7) [fill=purple] circle [radius=\rad];
\draw (A1-8) [fill=red!60] circle [radius=\rad];
\draw (A1-9) [fill=green] circle [radius=\rad];
\draw (A1-10) [fill=pink] circle [radius=\rad];
\draw (A1-11) [fill=blue!30] circle [radius=\rad];
\draw (A1-12) [fill=brown] circle [radius=\rad];
\draw (A1-13) [fill=gray] circle [radius=\rad];
\draw (A1-14) [fill=red] circle [radius=\rad];
\draw (A1-15) [fill=black!30] circle [radius=\rad];
\draw (A1-16) [fill=blue] circle [radius=\rad];



\draw ($(A0-0)-(1.75*\wi,0)+0.5*(0,\wi)-0.5*(\wi,\wi)-0.5*(\wi,0)$) node {\textbf{b)}};

\end{tikzpicture}
\hspace{0.25cm}
\begin{tikzpicture}[define rgb/.code={\definecolor{mycolor}{rgb}{#1}}, rgb color/.style={define rgb={#1},mycolor},scale=0.9]

\def\wi{0.25cm}
\def\spacer{2cm}
\def\n{16}
\def\rad{0.07cm}

\foreach \x in {0,1}
\foreach \y in {1,...,\n}
{
\coordinate (A\x-\y) at ($\n*(0,\wi)+\x*(\spacer,0)-\y*(0,\wi)+(0,\wi)$);
}

\draw [thick,yellow] (A0-1) -- (A1-1);
\draw [thick,teal] (A0-2) -- (A1-3);
\draw [thick,pink] (A0-3) -- (A1-10);
\draw [thick,blue!30] (A0-4) -- (A1-11);
\draw [thick,darkgreen!50] (A0-5) -- (A1-2);
\draw [thick,brown] (A0-6) -- (A1-12);
\draw [thick,red] (A0-7) -- (A1-14);
\draw [thick,violet] (A0-8) -- (A1-6);
\draw [thick,red!60] (A0-9) -- (A1-8);

\draw [thick,black!30] (A0-12) -- (A1-15);
\draw [thick,orange] (A0-13) -- (A1-4);
\draw [thick,green] (A0-14) -- (A1-9);
\draw [thick,blue] (A0-15) -- (A1-16);
\draw [thick,magenta] (A0-16) -- (A1-5);

\draw (A0-1) -- (A1-1);
\draw (A0-2) -- (A1-3);
\draw (A0-3) -- (A1-10);
\draw (A0-4) -- (A1-11);
\draw (A0-5) -- (A1-2);
\draw (A0-6) -- (A1-12);
\draw (A0-7) -- (A1-14);
\draw (A0-8) -- (A1-6);
\draw (A0-9) -- (A1-8);

\draw (A0-12) -- (A1-15);
\draw (A0-13) -- (A1-4);
\draw (A0-14) -- (A1-9);
\draw (A0-15) -- (A1-16);
\draw (A0-16) -- (A1-5);

\foreach \x in {0,1}
\foreach \y in {1,...,\n}
{
\draw (A\x-\y) [fill] circle [radius=\rad];
}

\draw (A1-1) [fill=yellow] circle [radius=\rad];
\draw (A1-2) [fill=darkgreen!50] circle [radius=\rad];
\draw (A1-3) [fill=teal] circle [radius=\rad];
\draw (A1-4) [fill=orange] circle [radius=\rad];
\draw (A1-5) [fill=magenta] circle [radius=\rad];
\draw (A1-6) [fill=violet] circle [radius=\rad];
\draw (A1-7) [fill=purple] circle [radius=\rad];
\draw (A1-8) [fill=red!60] circle [radius=\rad];
\draw (A1-9) [fill=green] circle [radius=\rad];
\draw (A1-10) [fill=pink] circle [radius=\rad];
\draw (A1-11) [fill=blue!30] circle [radius=\rad];
\draw (A1-12) [fill=brown] circle [radius=\rad];
\draw (A1-13) [fill=gray] circle [radius=\rad];
\draw (A1-14) [fill=red] circle [radius=\rad];
\draw (A1-15) [fill=black!30] circle [radius=\rad];
\draw (A1-16) [fill=blue] circle [radius=\rad];



\draw ($(A0-0)-(1.75*\wi,0)+0.5*(0,\wi)-0.5*(\wi,\wi)-0.5*(\wi,0)$) node {\textbf{c)}};

\end{tikzpicture}
\hspace{0.25cm}
\begin{tikzpicture}[define rgb/.code={\definecolor{mycolor}{rgb}{#1}}, rgb color/.style={define rgb={#1},mycolor},scale=0.9]

\def\wi{0.25cm}
\def\n{16}

\foreach \x in {0,1,...,\n}
\foreach \y in {0,1,...,\n}
{
\coordinate (A\x-\y) at ($\n*(0,\wi)+\x*(\wi,0)-\y*(0,\wi)$);
}


\foreach \x/\y/\col in {5/0/brown,11/0/black!30,14/0/blue!60,5/1/brown,11/1/black!30,14/1/blue!60,7/2/violet,11/2/black!30,14/2/blue!60,7/3/violet,11/3/black!30,14/3/blue!60,1/4/teal,3/4/blue!30,7/4/violet,15/4/magenta,1/5/teal,3/5/blue!30,7/5/violet,15/5/magenta,1/6/teal,3/6/blue!30,5/6/brown,15/6/magenta,1/7/teal,3/7/blue!30,5/7/brown,15/7/magenta,0/8/yellow,4/8/darkgreen!50,12/8/orange,13/8/green,0/9/yellow,4/9/darkgreen!50,12/9/orange,13/9/green,0/10/yellow,4/10/darkgreen!50,12/10/orange,13/10/green,0/11/yellow,4/11/darkgreen!50,12/11/orange,13/11/green,2/12/pink,6/12/red,8/12/red!60,2/13/pink,6/13/red,8/13/red!60,2/14/pink,6/14/red,8/14/red!60,2/15/pink,6/15/red,8/15/red!60}
{
\draw (A\x-\y) [thin,\col,fill=\col] rectangle ($(A\x-\y)+(\wi,\wi)$);
}

\foreach \x/\y in {2/15,6/12,8/14,0/8,4/11,12/9,13/10,1/5,3/4,5/7,7/3,11/1,14/0,15/6}
{
\draw (A\x-\y) -- ($(A\x-\y)+(\wi,\wi)$);
\draw ($(A\x-\y)+(\wi,0)$) -- ($(A\x-\y)+(0,\wi)$);
}

\draw ($(A0-0)-(1.75*\wi,0)+0.25*(0,\wi)$) node {\textbf{d)}};

\foreach \x in {0,1,...,\n}
{
{
\draw [black] ($(A\x-\n)+(0,\wi)$) --  ($(A\x-\n)+\n*(0,\wi)+(0,\wi)$);
\draw [black] ($(A0-\x)+(0,\wi)$) --  ($(A0-\x)+\n*(\wi,0)+(0,\wi)$);
}
}

\end{tikzpicture}
\end{center}
\caption{\textbf{a)} A subsquare $S$ in which each column can be decomposed into 4 blocks with size 4 which have the same colour (for ease of visualisation, cells of each colour are often drawn next to each other in natural blocks, but this is not required in general).
\textbf{b)} The multigraph $K$ corresponding to $S$. \textbf{c)} A matching $M$ in $K$.  \textbf{d)} The subsquare $S'\subset S$ corresponding to the matching $M$. Any set of cells which share no row or column must then share no colour and hence be a transversal; an example is marked by crosses.\label{fig:simplifiedexample}}
\end{figure}

Before discussing this in more detail, along with how we find an appropriate random matching, let us note that at a very high level this is inspired by work of the second author with Pokrovskiy and Sudakov on Ringel's conjecture~\cite{montgomery2021proof}. That is to say, when the codegrees are small, we can find a large transversal using semi-random methods. Where many colours appear in the same row or the same column, some indication of which colour we might wish to use for each column could be selected deterministically using some matching like the initial (non-random) matching $M$. Key (as in \cite{montgomery2021proof}) is to find an appropriate way to randomise this deterministic selection so that it can be used along with the semi-random method. Indeed, though proving Theorem~\ref{thm:lowerbounddecomp} in full will involve several more complications, this sketch encompasses the main novelty in our methods, which can then be developed into a more complex scheme using the semi-random method, as described and done in Section~\ref{sec:lowerbound}.

Let us return to our equi-$n$-square $S$ and the decomposition $\mathcal{B}$ of $S$ into blocks of size $m$. Let $\mathcal{J}$ be the set of columns and let $\mathcal{A}$ be the set of symbols. Let $K$ be an auxiliary bipartite multigraph with vertex classes $\mathcal{J}$ and $\mathcal{A}$, where, for each $B\in \mathcal{B}$, we add an edge labelled with $B$ between the column and the symbol of $B$. Note that $K$ is a 4-regular bipartite multigraph. Therefore, using Hall's matching criterion, it can be decomposed into 4 matchings, say $M_1,M_2,M_3,M_4$.

We will use $M_1,M_2,M_3,$ and $M_4$ to generate our random matching $M$, by finding random matchings $M_1'\subset M_1\cup M_2$ and $M_2'\subset M_3\cup M_4$, and then selecting $M\subset M_1'\cup M_2'$ randomly.
For example, for $M_1'$ the initial idea is to observe that $M_1\cup M_2$ is a union of cycles and, for each such cycle, randomly choose the `odd or even' edges of the cycle and add them to $M_1'$. This creates a matching, but if the cycles are long then there will be a lot of dependence between whether different edges appear in $M_1'$, or not, preventing us from showing that the number of entries in each row in blocks in $M_1'$ is suitably concentrated.
The key idea then is to take $M_1\cup M_2$ and delete a few edges to break the cycles into paths/cycles which are each not too long. Doing similarly with $M_3\cup M_4$ to generate $M_2'$, we then similarly generate $M$ from $M_1'\cup M_2'$ (which is the disjoint union of paths/cycles rather than just cycles). By choosing the maximum length of the paths/cycles allowed after the edge deletion carefully, we can (as described below) ensure that $M$ is a large (almost-perfect) matching while limiting the dependence over the appearance of different edges in $M$ with each other. We then track certain Lipschitz random variables and show that they are likely to be concentrated.

More specifically, let $k=n^{1/3}\log^{-2}n$. Let $S_1,\ldots, S_r$ be a minimal set of vertex-disjoint paths/cycles in $M_1\cup M_2$, each with length at most $k$, the union of whose vertex sets is $V(M_1\cup M_2)$. Note that $|E(M_1\cup M_2)\setminus E(S_1\cup \ldots \cup S_r)|\leq 2n/k$. For each $i\in [r]$, properly edge colour $S_1,\ldots,S_r$ with red and blue (noting all the cycles are even). For each $i\in [r]$, choose $x_i$ uniformly and independently at random from $\{0,1\}$. Let $M_1'$ be the union across $i\in [r]$ of the red edges from $E(S_i)$ with $x_i=0$ and the blue edges of $E(S_i)$ with $x_i=1$. Observe that $M_1'$ is a matching.
Similarly, use $M_3$ and $M_4$ to randomly generate $M_2'$. Then,  similarly use $M_1'\cup M_2'$ to generate $M$.

We use $M$ to tell us from which column we will select a cell with which symbol in it (and which symbols we will not select a cell for, but these will be a small number of symbols overall).
 For this, let $\mathcal{I}$ be the set of $n$ rows. Form the auxiliary bipartite graph $L$ on vertex classes $\mathcal{I}$ and $\mathcal{J}$, where there is an edge $ij$ exactly if $j$ is matched to a symbol with an edge in $M$ corresponding to a block containing the cell $(i,j)$. 
 Observe, as $M$ is a matching, that a matching in $L$ corresponds to a transversal in $S$.

 For each row $i\in \mathcal{I}$, let $m_1(i)$ be the number of edges $jc\in M_1$ which correspond to a block containing a cell on row $i$. 
 Define $m_2(i)$, $m_3(i)$, $m_4(i)$, $m_1'(i)$, $m_2'(i)$ and $m(i)$ similarly. Note that $m_1(i)+m_2(i)+m_3(i)+m_4(i)=n$ and $\E(m_1'(i))\leq (m_1(i)+m_2(i))/2$. As $|E(M_1\cup M_2)\setminus E(S_1\cup \ldots \cup S_r)|\leq 2n/k$, we have
\[
\E (m_1'(i))\geq \frac{m_1(i)+m_2(i)}{2}-\frac{2n}{k}\geq \frac{m_1(i)+m_2(i)}{2}-n^{2/3}\log^{O(1)}n.
\]
Crucially, changing the value of any one of $x_i$, $i\in [r]$, changes the value of $m_1'(i)$ by at most $k=n^{1/3}\log^{-2}n$. Thus, by an application of McDiarmid's inequality (Lemma~\ref{lem:mcdiarmid}), we get that, with probability $1-o(n^{-1})$,
\[
m_1'(i)=\frac{m_1(i)+m_2(i)}{2}\pm \left(n^{2/3}\log^{O(1)}n+O\left(\sqrt{n k\log n}\right)\right)=\frac{m_1(i)+m_2(i)}{2}\pm n^{2/3}\log^{O(1)}n.
\]
Similarly, with probability $1-o(n^{-1})$, $m_2'(i)=\frac{m_3(i)+m_4(i)}{2}\pm n^{2/3}\log^{O(1)}n$. Similarly, then, with probability $1-o(n^{-1})$,
\[
d_L(i)=m(i)=\frac{m_1(i)+m_2(i)+m_3(i)+m_4(i)}{4}\pm n^{2/3}\log^{O(1)}n=\frac{n}{4}\pm n^{2/3}\log^{O(1)}n=m\pm n^{2/3}\log^{O(1)}n.
\]

For each column $j\in \mathcal{J}$ with $j$ in some edge in $M$, we have $d_L(j)=m=n/4$, and otherwise $d_L(j)=0$. Then, for each $U\subset \mathcal{I}$, we have $(m-n^{2/3}\log^{O(1)}n)\cdot |U|\leq m\cdot |N_L(U)|$, so that
\[
|N_L(U)|\geq |U|\left(1-\frac{n^{2/3}\log^{O(1)}n}{m}\right)\geq |U|-n^{2/3}\log^{O(1)}n.
\]
Thus, by a defect version of Hall's matching theorem, $L$ contains a matching with size $n-n^{2/3}\log^{O(1)}n$, and hence $S$ has a transversal with $n- n^{2/3}\log^{O(1)}n$ cells.


\section{{\boldmath Almost decomposing equi-$n$-squares}}\label{sec:lowerbound}

We will use the following result of Molloy and Reed~\cite{molloy2000near}.

\begin{theorem}\label{thm:molloyreed}
For all $k$ there is a constant $C_k$ such that any $k$-uniform hypergraph of maximum codegree $B$ and maximum degree $\Delta$ has list chromatic index at most $\left(1+C_k(B/\Delta)^{1/k}(\log (\Delta/B))^4\right)\Delta$.
\end{theorem}

Specifically, we will use the following corollary. 

\begin{corollary}\label{cor:molloyreed}
For all $\eps >0$ there is a $\mu >0$ and $D_0$ such that the following holds for each $D \geq D_0$. Suppose $\mathcal{H}$ is a $3$-uniform hypergraph with maximum degree at most $D$ such that, for each distinct $x,y\in V(\mathcal{H})$, either $\mathrm{cod}_{\mathcal{H}}(x,y)\leq D^{1-\eps}$ or every edge contains either both $x$ and $y$ or none of them.

Then, $\mathcal{H}$ has chromatic index at most $\left(1+D^{-\mu}\right)D$.
\end{corollary}
\begin{proof} Let $0<1/D_0\ll\mu\ll \eps$.
Let $E=\{xy:\mathrm{cod}_{\mathcal{H}}(x,y)>D^{1-\eps}\}$, and note that $E$ is a matching. For each $f\in E$, pick $x_f\in V(f)$. For each $f\in E$ and $e\in E(\mathcal{H})$ with $V(f)\subset V(e)$ create a new vertex, $v_e$. Let $\mathcal{H}'$ be the hypergraph with vertex set $(V(\mathcal{H})\setminus\{x_f:f\in E\})\cup (\cup_{f\in E}\cup_{e\in E(\mathcal{H})}v_e)$ and edge set
 \[
(E(\mathcal{H})\setminus (\cup_{f\in E}\cup_{e\in E(\mathcal{H}):V(e)\subset V(f)}e))\cup (\cup_{f\in E}\cup_{e\in E(\mathcal{H}):V(e)\subset V(f)}(\{v_e\}\cup (V(e)\setminus \{x_f\}))).
\]
 Observe that $\Delta(\mathcal{H}')\leq D$ and $\mathrm{cod}_{\mathcal{H}'}(x,y)\leq D^{1-\eps}$ for each $x,y\in V(\mathcal{H}')$. Then, by Theorem~\ref{thm:molloyreed}, we have $\chi'(\mathcal{H}')\leq \left(1+D^{-\mu}\right)D$.
 Let $\chi'$ be a colouring of $E(\mathcal{H}')$ with at most $\left(1+D^{-\mu}\right)D$ colours so that the edges of each colour forms a matching. For each $e\in E(\mathcal{H}')\cap E(\mathcal{H})$ let $\chi(e)=\chi'(e)$. For each $e\in E(\mathcal{H})\setminus E(\mathcal{H}')$, take $f\in E$ such that $V(f)\subset V(e)$ and let $\chi(e)=\chi'((V(e)\setminus \{x_f\})\cup \{v_e\})$. Observing that $\chi$ is a colouring of $E(\mathcal{H})$ into at most $\left(1+D^{-\mu}\right)D$ colours in which the edges of each colour form a matching, completes the proof.
\end{proof}

We will use this with our bounded-dependence random matching algorithm to prove the following.

\begin{theorem}\label{thm:semiunifieddecomp} There are some $\xi,\eta>0$ and $D_0$ such that the following holds for each $D\geq D_0$. Let $\mathcal{H}$ be a $3$-uniform hypergraph with maximum degree $D$ and at least $\sqrt{D}$ vertices. Suppose that the graph on $V(\mathcal{H})$ with edges $xy$ present if $\mathrm{cod}_{\mathcal{H}}(x,y)\geq D^{1-\eta}$ is bipartite.

Then, there is some $\mathcal{H}'\subset \mathcal{H}$ with $e(\mathcal{H}')\geq e(\mathcal{H})-|\mathcal{H}|\cdot D^{1-\xi}$ and
\[
\chi'(\mathcal{H}')\leq (1+D^{-\xi})D.
\]
\end{theorem}

We will then deduce Theorem~\ref{thm:lowerbounddecomp} from this.
We will use McDiarmid's inequality, in the following form (see \cite[Lemma~1.2]{M89}).
\begin{lemma}\label{lem:mcdiarmid} Let $n\in \N$ and $c_1,\ldots,c_n\geq 0$. For each $i\in [n]$, let $X_i$ be an independent random variable taking values in $\Omega_i$, and let $X=(X_1,\ldots,X_n)$.
Let $f:\prod_{i=1}^n\Omega_i\to \mathbb{R}$ be a function such that, for each $i\in [n]$, changing $X$ in the $i$th co-ordinate changes the value of $f(X)$ by at most $c_i$.

Then, for all $t>0$,
\[
\P(|f(X)-\mathbb{E}(f(X))|>t)\leq 2\exp\left(-\frac{t^2}{\sum_{i\in [n]}c_i^2}\right).
\]

\end{lemma}


\subsection{A random subhypergraph with `all or small' codegrees}
To prove Theorem~\ref{thm:semiunifieddecomp}, we will first use our methods to prove the following lemma.

\begin{lemma}\label{lem:semiunifieddecomp} There are some $\eps,\mu,\eta>0$ and $D_0$ such that the following holds for each $D\geq D_0$. Let $\mathcal{H}$ be a $3$-uniform hypergraph with maximum degree $D$ and at least $\sqrt{D}$ vertices. Suppose that the graph on $V(\mathcal{H})$ with edges $xy$ present if $\mathrm{cod}_{\mathcal{H}}(x,y)\geq D^{1-\eta}$ is bipartite.

Then, there is some $D^{3\mu}\leq t\leq 8D^{3\mu}$ and a random subhypergraph $\mathcal{H}'\subset \mathcal{H}$ such that
\stepcounter{propcounter}
\begin{enumerate}[label = {\emph{\textbf{\Alph{propcounter}\arabic{enumi}}}}]
\item for each $e\in E(\mathcal{H})$, $\P(e\in E(\mathcal{H}'))\leq 1/t$,\labelinthm{prop:noedgethatlikely}
\item $\Delta(\mathcal{H}')\leq (1+D^{-\eps})D/t$,\labelinthm{prop:maxdegree}
\item for each distinct $x,y\in V(\mathcal{H})$, either $\mathrm{cod}_{\mathcal{H}'}(x,y)\leq D^{1-\eps}/t$ or every $e\in E(\mathcal{H}')$ satisfies $\{x,y\}\subset V(e)$ or $\{x,y\}\cap V(e)=\emptyset$,\labelinthm{prop:allorlittlecodegree}
\end{enumerate}
and, with probability at least $1-D^{-10}$, the following holds.
\begin{enumerate}[label = {\emph{\textbf{\Alph{propcounter}\arabic{enumi}}}}]\addtocounter{enumi}{3}
\item $e(\mathcal{H'})\geq e(\mathcal{H})/t-D^{1-\eps}|\mathcal{H}|/t$.\labelinthm{prop:manyedges}
\end{enumerate}
\end{lemma}
\begin{proof}
Let $0 < 1/D_0\ll \eps\ll \gamma\ll \mu\ll \eta \ll 1$ and $D\geq D_0$. (Probably, we can simply take $\eta = 2\mu$.) Let $\mu$ be such that $k=D^{\mu}$ is a power of $2$, let $\ell\in \mathbb{N}$ then be such that $k=2^\ell$, and let $t=k^3$.
Let $V=V(\mathcal{H})$ and $m=\lfloor D/k\rfloor$. As set-up, do the following.

\stepcounter{propcounter}
\begin{enumerate}[label = {{\textbf{\Alph{propcounter}\arabic{enumi}}}}]
\item Remove from $\mathcal{H}$ any edge $xyz$ in which $\{x,y\}$ and $\{y,z\}$ both have codegree at least $m$, and let the result be $\mathcal{H}_0$.\label{prop:removedoublecodegreeedges}
\item Take a maximal collection $\mathcal{B}$ of disjoint subsets of $E(\mathcal{H}_0)$ with size $m$ such that, for every $B\in \mathcal{B}$, there is some set $U_B\subset V$ with $|U_B|=2$ and $U_B\subset V(e)$ for each $e\in B$.\label{prop:initialcollectionB}
\item Let $\mathcal{B}_v$, $v\in V$, maximise $\sum_{v\in V}|\mathcal{B}_v|$ subject to the following conditions.\label{prop:initialcolletionsBv}
\begin{enumerate}[label = {{\textbf{\Alph{propcounter}\arabic{enumi}.\arabic{enumii}}}}]
\item For each $v\in V$, the sets in $\mathcal{B}_v$ are disjoint subsets of
$\{e\in E(\mathcal{H}_0):v\in V(e)\}\setminus \left(\cup_{B\in \mathcal{B}:v\in U_B}B\right)$ 
with size $m$.\label{prop:origdegreebound}
\item For each distinct $v,w\in V$ and $B\in \mathcal{B}_v$, the number of edges in $B$ containing $w$ is at most $D^{2\gamma}m/k$.\label{prop:origcodegreebound}
\item For each distinct $v,w\in V$ and $B\in \mathcal{B}_v$, $B'\in \mathcal{B}_w$, $|B\cap B'|\leq D^\gamma m/k^2$.\label{prop:origcodegreepairbound}
\end{enumerate}
\end{enumerate}



Let $\mathcal{H}_1$ be the subhypergraph of $\mathcal{H}_0$ which contains exactly the edges $e=uvw\in E(\mathcal{H}_0)$ for which
either there is some $B\in \mathcal{B}$ with $e\in B$ and some $v\in V$ and $B'\in \mathcal{B}_v$ with $e\in B'$ or, for each $v\in V(e)$, there is some $B\in \mathcal{B}_v$ with $e\in B$. For each $B\in \mathcal{B}$, let $A_B^0=B$.

Let $K$ be an auxiliary multigraph with vertex set $V\cup (\cup_{v\in V}\mathcal{B}_v)$ where, for each $B\in \mathcal{B}$, we add $U_B$ to $E(K)$ and, for each $v\in V$ and $B\in \mathcal{B}_v$, we add $vB$ to $E(K)$. Note that, by the conditions on $\mathcal{H}$, we have that $K$ is a bipartite multigraph with maximum degree at most $k$, and therefore (for example, by embedding $K$ into a $k$-regular bipartite multigraph and repeatedly using Hall's matching criterion to find a perfect matching) $\chi'(K)\leq k$. Thus, we can partition $K$ into matchings $M^0_1,\ldots,M^0_k$. Let $s=D^{\gamma/2}$.


 For each $1\leq h\leq \ell$ in turn do the following with the matchings $M^{h-1}_1,\dots,M^{h-1}_{k/2^{h-1}}$ to generate matchings $M^h_1,\ldots,M^h_{k/2^h}$ randomly:
\begin{itemize}
\item Let $F_h\subset M^{h-1}_1\cup \dots\cup M^{h-1}_{k/2^{h-1}}$ be a minimal set of edges such that, for each $i\in [k/2^h]$, if, for the appropriate $r_{i}^h$, $M^{h-1}_{2i}\cup M^{h-1}_{2i-1}-F_h$ (considered as a multigraph and removing any copies of edges in $F_h$) has as its components the paths/cycles $S^h_{i,1},\ldots,S^h_{i,r_i^h}$, then, each such path/cycle has length at most $s$.
\item For each $i\in [k/2^h]$, and each $j\in [r_i^h]$, pick $x_{i,j}^h$ uniformly and independently at random from $\{0,1\}$ and let
\[
M^{h}_i=\left(\cup_{j\in [r_i^h]:x_{i,j}^h=1}E(S^h_{i,j})\cap M^{h-1}_{2i}\right)\cup \left(\cup_{j\in [r_i^h]:x_{i,j}^h=0}E(S^h_{i,j})\cap M^{h-1}_{2i-1}\right).
\]
\item For each $B\in \mathcal{B}$, let $A_B^h\subset A_B^{h-1}$ be chosen by including each element of $A_B^{h-1}$ independently at random with probability $1/2$.
\end{itemize}

Let $M=M^\ell_1$ and $\mathcal{H}^0_1=\mathcal{H}^0_2=\mathcal{H}_1$. For each $h\in [\ell]$, let $M^h=\cup_{i\in [k/2^h]}M^h_i$. For each $h\in [\ell]$, let $\mathcal{H}^h_0$
be a subhypergraph of $\mathcal{H}^{h-1}_2$ which maximises $e(\mathcal{H}^h_0)$ subject to the following conditions.
\stepcounter{propcounter}
\begin{enumerate}[label = {{\textbf{\Alph{propcounter}\arabic{enumi}}}}]

\item For each $B\in \mathcal{B}$ and each $v\in V$ and $B'\in \mathcal{B}_v$, if $(B\cap B')\cap E(\mathcal{H}^{h}_0)\neq \emptyset$, then $U_B,vB'\notin F^h$ and $U_B$ and $vB'$ do not appear in any path/cycle in $S^h_{i,1},\ldots,S^h_{i,r_i^h}$ together, for any $i\in [k/2^h]$.\label{prop:notfromsamecycle2}
\item For each distinct $v,w\in V$ and $B\in \mathcal{B}_v$, $B'\in \mathcal{B}_w$, if $(B\cap B')\cap E(\mathcal{H}^{h}_0)\neq \emptyset$, then $vB,wB'\notin F^h$ and $vB$ and $wB'$ do not appear in any path/cycle in $S^h_{i,1},\ldots,S^h_{i,r_i^h}$ together, for any $i\in [k/2^h]$.\label{prop:notfromsamecycle1}
\end{enumerate}

Then, for each $h\in [\ell]$, let $\mathcal{H}^h_1$ be the hypergraph with vertex set $V$ and edges $e\in E(\mathcal{H}^h_0)$ for which either there is some $B\in \mathcal{B}$ with $e\in B$ and some $v\in V$ and $B'\in \mathcal{B}_v$ with $e\in B'$, where $U_B,vB'\in M^h$ and $e\in A^h_B$, or, for each $v\in V(e)$, there is some $B\in \mathcal{B}_v$ with $e\in B$ and $vB\in M^h$.

Then, for each $h\in [\ell]$, let $\mathcal{H}^h_2$ be a subhypergraph of $\mathcal{H}^h_1$ which maximises $e(\mathcal{H}^h_2)$ subject to the following conditions.
\stepcounter{propcounter}
\begin{enumerate}[label = {{\textbf{\Alph{propcounter}\arabic{enumi}}}}]
\item For each $B\in \mathcal{B}\cup (\cup_{v\in V}\mathcal{B}_v)$, $|B\cap E(\mathcal{H}^h_2)|\leq (1+h\cdot D^{-\eps}/\ell) m/2^{2h}$.\label{prop:degreedecreasing}
\item For each distinct $v,w\in V$ and $B\in \mathcal{B}_v$, the number of edges in $B\cap E(\mathcal{H}^h_2)$ containing $w$ is at most $(1+h/\ell) D^{2\gamma} m/2^{2h}k$.\label{prop:codegreedecreasing}
\item For each distinct $v,w\in V$ and $B\in \mathcal{B}_v$, $B'\in \mathcal{B}_w$, $|(B\cap B')\cap E(\mathcal{H}^h_2)|\leq (1+h/\ell) D^\gamma m/2^{h}k^2$.\label{prop:codegreepairsdecreasing}
\end{enumerate}

Finally, let $\mathcal{H}'=\mathcal{H}^\ell_2$.
We will show that $\mathcal{H}'$ satisfies \ref{prop:noedgethatlikely}--\ref{prop:allorlittlecodegree}, and, with probability at least $1-D^{-10}$, \ref{prop:manyedges} holds.

\medskip


\noindent\ref{prop:noedgethatlikely}: \textbf{Case 1.} Let $e\in E(\mathcal{H})$ be such that $e\in B_e$ for some $B_e\in \mathcal{B}$ and there is some $v_e\in V$ and $B'_e\in \mathcal{B}_{v_e}$ with $e\in B'_e$.
For each $h\in [\ell]$, we will show that $\P(e\in E(\mathcal{H}_1^h)|e\in E(\mathcal{H}_0^h))= 1/8$. Suppose then that $h\in [\ell]$ and $e\in E(\mathcal{H}_0^h)$. Note that, as $e\in E(\mathcal{H}_0^h)$, from \ref{prop:notfromsamecycle2} we have that $U_{B_e}$ and $v_{e}B'_{e}$ are not in the same path/cycle $S_{i,j}^h$ for any $i\in [k/2^h]$ and $j\in [r_i^h]$.
Furthermore, $e\in E(\mathcal{H}_1^h)$ only if $U_{B_e}\in M^{h}$, $v_eB_e'\in M^h$ and $e\in A^h_{B_e}$, which are three independent events which each occur with probability $1/2$. Thus, $\P(e\in E(\mathcal{H}_1^h)|e\in E(\mathcal{H}_0^h))=1/8$.
Therefore, $\P(e\in E(\mathcal{H}'))\leq \P(e\in E(\mathcal{H}_1^\ell))\leq 1/8^{\ell}=1/k^3=1/t$.

\smallskip

\noindent \textbf{Case 2.} Let $e\in E(\mathcal{H})$ be such that for each $v\in V(e)$ there is some $B_{e,v}\in \mathcal{B}_v$ with $e\in B_{e,v}$.
For each $h\in [\ell]$, we will show that $\P(e\in E(\mathcal{H}_1^h)|e\in E(\mathcal{H}_0^h))= 1/8$. Suppose then that $h\in [\ell]$ and $e\in E(\mathcal{H}_0^h)$. Note that, as $e\in E(\mathcal{H}_0^h)$, from \ref{prop:notfromsamecycle1} we have that no two of $vB_{e,v}$ are in the same path/cycle $S_{i,j}^h$ for any $i\in [k/2^h]$ and $j\in [r_i^h]$.
Furthermore, $e\in E(\mathcal{H}_1^h)$ only if $vB_{e,v}\in M^{h}$ for each $v\in V(e)$, which are three independent events which each occur with probability $1/2$. Thus, $\P(e\in E(\mathcal{H}_1^h)|e\in E(\mathcal{H}_0^h))=1/8$.
Therefore, $\P(e\in E(\mathcal{H}'))\leq \P(e\in E(\mathcal{H}_1^\ell))\leq 1/8^{\ell}=1/k^3=1/t$.


\medskip

\noindent\ref{prop:maxdegree}: Let $v\in V$. Note that if there is no edge in $M$ that contains $v$ then there are no edges in $\mathcal{H}'$ containg $v$, and thus $d_{\mathcal{H}'}(v)=0$.
As there is at most one edge in $M$ which contains $v$, we can assume that there is exactly one such edge. Suppose that this edge corresponds to $U_B$ for some $B\in \mathcal{B}$. Then, the only edges $e\in E(\mathcal{H}')$ with $v\in V(e)$ are those in $B$, and therefore,  by \ref{prop:degreedecreasing},
\[
d_{\mathcal{H}'}(v)\leq |B\cap E(\mathcal{H}_2^\ell)|\leq (1+D^{-\eps})m/2^{2\ell}=(1+D^{-\eps})D/k^3=(1+D^{-\eps})D/t.
\]
Suppose then that the edge in $M$ containing $v$ is $vB$ for some $B\in \mathcal{B}_v$. Then, similarly by \ref{prop:degreedecreasing}, $d_{\mathcal{H}'}(v)\leq |B\cap E(\mathcal{H}_2^\ell)|\leq (1+D^{-\eps})D/t$.
Thus, \ref{prop:maxdegree} holds.


\medskip

\noindent\ref{prop:allorlittlecodegree}: Let $x,y\in V(\mathcal{H})$ be distinct. Suppose there is some $B\in \mathcal{B}$ with $U_B=\{x,y\}$ and $U_B\in M$. Let $e\in E(\mathcal{H}')$ contain $x$. Then, $M$ must contain an edge which contains $x$ which is either $U_{B'}$ for some $B'\in \mathcal{B}$ or 
$xB'$ for some $B'\in \mathcal{B}_x$, for which $e\in B'$. As $M$ is a matching, the only possibility is $B'=B$, and then $\{x,y\}\subset V(e)$. Thus, there are no edges in $\mathcal{H}'$ that contain $x$ but not $y$.
 Arguing similarly, there are also no edges in $\mathcal{H}'$ that contain $y$ but not $x$. Thus, every edge in $\mathcal{H}''$ contains either both $x$ and $y$ or neither of them.

Now, suppose there is no $B\in \mathcal{B}$ with $U_B=\{x,y\}$ and $U_B\in M$. If there is some $B\in \mathcal{B}$ with $x\in U_B$ and $U_B\in M$, then, if $z$ is such that $U_B=\{x,z\}$, then the only possible edge in $\mathcal{H}'$ containing $x$ and $y$ is $xyz$, and thus $\mathrm{cod}_{\mathcal{H}'}(x,y)\leq 1$.
Similarly, if there is some $B\in \mathcal{B}$ with $y\in U_B$ and $U_B\in M$, then $\mathrm{cod}_{\mathcal{H}'}(x,y)\leq 1$.
Therefore, if $\mathrm{cod}_{\mathcal{H}'}(x,y)> 1$, there must be some $B\in \mathcal{B}_x$ and $B'\in \mathcal{B}_y$ with $xB,yB'\in M$, so that, from \ref{prop:codegreepairsdecreasing},
$\mathrm{cod}_{\mathcal{H}'}(x,y)\leq |(B\cap B')\cap E(\mathcal{H}^\ell_2)|\leq 2D^\gamma m/2^{\ell}k^2= 2D^{\gamma-\mu}D/t\leq  D^{1-\eps}/t$, where we have used that $\eps\ll \gamma\ll \mu$. This completes the proof of \ref{prop:allorlittlecodegree}.

\medskip

\noindent\ref{prop:manyedges}: It is left then only to prove that \ref{prop:manyedges} holds with probability at least $1-D^{-10}$.
We start by showing the following claim.



\begin{claim} \label{clm:codegreepairsdecreasing}
For each $h\in [\ell]$ and each distinct $v,w\in V$ and $B\in \mathcal{B}_v$, $B'\in \mathcal{B}_w$, with probability at least $1-D^{-20}$,
\begin{equation}\label{eqn:smallcodegreepairs}
|(B\cap B')\cap E(\mathcal{H}^h_1)|\leq (1+h/\ell) D^\gamma m/2^{h}k^2.
\end{equation}
\end{claim}
\begin{proof}[Proof of Claim~\ref{clm:codegreepairsdecreasing}] Let $h\in [\ell]$. Let $v,w\in V$ be distinct and let $B\in \mathcal{B}_v$ and $B'\in \mathcal{B}_w$. Now, by the choice of $\mathcal{H}_2^{h-1}$ (in particular \ref{prop:origcodegreepairbound} or \ref{prop:codegreepairsdecreasing}), we always have that
\begin{equation}\label{eqn:codegreepair}
|(B\cap B')\cap E(\mathcal{H}^{h}_0)|\leq |(B\cap B')\cap E(\mathcal{H}^{h-1}_2)|\leq (1+(h-1)/2\ell)D^\gamma m/2^{h-1}k^2.
\end{equation}
Let $Z_{v,w}$ be the set of $z$ such that $vwz\in (B\cap B')\cap E(\mathcal{H}^{h}_0)$. For each $z\in Z_{v,w}$, let $B_z\in \mathcal{B}_z$ be such that $vwz\in B_z$, noting that, therefore, $zB_z\in M^{h-1}$. For each $z\in Z_{v,w}$,
note that, if $vwz\in E(\mathcal{H}^h_1)$, then we must have that the indicator variable for the cycles/paths containing $zB_z$ must be equal to 1 or 0 depending on which matching $zB_z$ is in in $M^{h-1}=M_1^{h-1}\cup\ldots\cup M^{h-1}_{k/2^{h-1}}$. Therefore, using \eqref{eqn:codegreepair},
\[
\E|(B\cap B')\cap E(\mathcal{H}^h_1)|\leq (1+(h-1)/2\ell) D^\gamma m/2^{h}k^2.
\]
For each $i\in [k/2^h]$ and $j\in [r_{i}^h]$, let $c_{i,j}^h$ be the number of edges in $S_{i,j}^h$ which are $zB_z$ for some $z\in Z_{v,w}$, and note that changing the variable $x_{i,j}^h$ changes $|(B\cap B')\cap E(\mathcal{H}^h_1)|$ by at most $c_{i,j}^h$, and that $c_{i,j}^h\leq s$.
Thus,
\[
\sum_{i\in [k/2^h]}\sum_{j\in [r_i^h]}(c_{i,j}^h)^2\leq s\cdot \sum_{i\in [k/2^h]}\sum_{j\in [r_i^h]}c_{i,j}^h\leq s|Z_{v,w}|\overset{\eqref{eqn:codegreepair}}{\leq} 2sD^\gamma m/2^{h-1}k^2.
\]
Therefore, by Lemma~\ref{lem:mcdiarmid}, \eqref{eqn:smallcodegreepairs} does not hold with probability at most
\[
2\exp\left(-\frac{(\frac{1}{\ell}D^\gamma m/2^{h}k^2)^2}{2sD^\gamma m/2^{h-1}k^2}\right)=2\exp\left(-\frac{D^\gamma m/2^{h}k^2}{4s\ell^2}\right)\leq 2\exp\left(-\frac{D^{1+\gamma}}{4s\ell^2k^4}\right)\leq D^{-20},
\]
as required.
\renewcommand{\qedsymbol}{$\boxdot$}
\end{proof}
\renewcommand{\qedsymbol}{$\square$}
Now we show a similar claim to Claim~\ref{clm:codegreepairsdecreasing} for codegrees.


\begin{claim} \label{clm:codegreedecreasing}
For each $h\in [\ell]$ and each distinct $v,w\in V$ and $B\in \mathcal{B}_v$, with probability at least $1-D^{-20}$,
the number of edges in $B\cap E(\mathcal{H}^h_1)$ containing $w$ is at most $(1+h/\ell) D^{2\gamma} m/2^{2h}k$.
\end{claim}
\begin{proof}[Proof of Claim~\ref{clm:codegreedecreasing}] Let $h\in [\ell]$. Let $v,w\in V$ be distinct and let $B\in \mathcal{B}_v$. Now, by the choice of $\mathcal{H}_2^{h-1}$ (in particular \ref{prop:origcodegreebound} or \ref{prop:codegreedecreasing}), we always have that the number of edges in $B\cap E(\mathcal{H}^h_0)$ containing $w$ is at most $(1+(h-1)/\ell) D^{2\gamma}m/2^{2(h-1)}k$.

Let $Z_{v,w}$ be the set of $z\in V\setminus \{v,w\}$ such that $vwz\in B\cap E(\mathcal{H}^{h}_0)$ and there is some $B_z\in \mathcal{B}$ such that $vwz\in B_z$ and $z\in U_{B_z}$. Note that, as $vwz\in B$, then, by \ref{prop:origdegreebound}, $U_{B_z}=\{z,w\}$.
Let  $Z_{v,w}'$ be the set of $z\in V\setminus \{v,w\}$ such that $vwz\in B\cap E(\mathcal{H}^{h}_0)$ and there is some $B_z\in \mathcal{B}_z$ with $vwz\in B_z$ and some $B_{z}^w\in \mathcal{B}_w$ with $vwz\in B_z^w$.
Note that $|Z_{v,w}|+|Z'_{v,w}|\leq (1+(h-1)/\ell) D^{2\gamma}m/2^{2(h-1)}k$.

Now, for each $z\in Z_{v,w}$, we have $U_{B_z}\in M^{h-1}$ and $U_{B_z}\notin F^h$ by the choice of $\mathcal{H}_0^h$ (from \ref{prop:notfromsamecycle2}). Note that if $vwz\in E(\mathcal{H}^h_1)$ then we have $vwz\in A^h_{B_z}$ and the indicator variable $x_{i,j}^h$ for the path/cycle containing $U_{B_z}$ must be equal to 0 or 1 depending on which matching $U_{B_z}$ is in in the partition $M^{h-1}=M_1^{h-1}\cup\ldots\cup M^{h-1}_{k/2^{h-1}}$. Thus, $vwz\in E(\mathcal{H}^h_1)$ with probability $1/4$.

Furthermore, for each $z\in Z'_{v,w}$, from \ref{prop:notfromsamecycle1} we have $zB_z,wB_z^w\in M^{h-1}\setminus F^h$, and that $zB_z$ and $wB_z^w$ do not appear in the same path/cycle in $S_{i,j}^h$, $i\in [k/2^h]$ and $j\in [r_i^h]$.
Note that if $vwz\in E(\mathcal{H}^h_1)$ then the indicator variables $x_{i,j}^h$ for the path/cycles containing $zB_z$ and $wB_z^w$ must be equal to 0 or 1 depending on which matching $zB_z$ and $wB_z^w$ is in respectively in the partition $M^{h-1}=M_1^{h-1}\cup\ldots\cup M^{h-1}_{k/2^{h-1}}$. Thus, $vwz\in E(\mathcal{H}^h_1)$ with probability $1/4$.

Therefore, altogether, we have
\[
\E|\{e\in B\cap E(\mathcal{H}_1^h):w\in V(e)\}|= \frac{1}{4}(|Z_{v,w}|+|Z_{v,w}'|)\leq (1+(h-1)/2\ell) D^{2\gamma} m/2^{2h}k.
\]

Now, for each $i\in [k/2^h]$ and $j\in [r_i^h]$, let $c^h_{i,j}$ be the number of $z\in Z_{v,w}$ with $U_{B_z}$ in the path/cycle $S^h_{i,j}$ or $z\in Z_{v,w}'$ with $zB_z$ or $wB_z^w$ in the path/cycle $S^h_{i,j}$. Then, changing the variable $x_{i,j}^h$ changes $|\{e\in B\cap E(\mathcal{H}_1^h):w\in V(e)\}|$ by at most $c_{i,j}^h$. Note that we have $c^h_{i,j}\leq s\cdot 2D^{\gamma} m/2^{h-1}k^2$ by the choice of $\mathcal{H}_2^{h-1}$ (in particular \ref{prop:origcodegreepairbound} or \ref{prop:codegreepairsdecreasing}).
Thus,
\begin{align*}
\sum_{i\in [k/2^h]}\sum_{j\in [r_{i}^h]}(c_{i,j}^h)^2&\leq (s\cdot 2D^{\gamma} m/2^{h-1}k^2)\cdot \sum_{i\in [k/2^h]}\sum_{j\in [r_{i}^h]}c_{i,j}^h\\
&\leq (s\cdot 2D^{\gamma} m/2^{h-1}k^2)\cdot (|Z_{v,w}|+2|Z_{v,w}'|)\\
&\leq (s\cdot 2D^{\gamma} m/2^{h-1}k^2)\cdot 4D^{2\gamma}m/2^{2(h-1)}k=8sD^{3\gamma}m^2/2^{3(h-1)}k^3.
\end{align*}
Furthermore, for each $z\in Z_{v,w}$, changing whether or not $vwz\in A_{B_z}^h$ changes $|\{e\in B\cap E(\mathcal{H}_1^h):w\in V(e)\}|$ by at most 1, and $|Z_{v,w}|+\sum_{i\in [k/2^h]}\sum_{j\in [r_{i}^h]}(c_{i,j}^h)^2\leq 16sD^{3\gamma}m^2/2^{3(h-1)}k^3$.
Thus, by Lemma~\ref{lem:mcdiarmid}, the number of edges in $B\cap E(\mathcal{H}^h_1)$ containing $w$ is more than $(1+h/\ell) D^{2\gamma} m/2^{2h}k$ with probability at most
\[
2\exp\left(-\frac{(\frac{1}{\ell}D^{2\gamma} m/2^{2h}k)^2}{16sD^{3\gamma}m^2/2^{3(h-1)}k^3}\right)
=2\exp\left(-\frac{D^{\gamma}k}{128s\ell^22^h}\right)\leq 2\exp\left(-\frac{D^{\gamma/2}}{128\ell^2}\right)\leq D^{-20},
\]
where we have used that $s=D^{\gamma/2}$ and $2^h\leq 2^\ell=k$.
\renewcommand{\qedsymbol}{$\boxdot$}
\end{proof}
\renewcommand{\qedsymbol}{$\square$}

Now we show a similar claim to Claim~\ref{clm:codegreedecreasing} for degrees.


\begin{claim} \label{clm:degreedecreasing}
For each $B\in \mathcal{B}\cup (\cup_{v\in V}\mathcal{B}_v)$, with probability at least $1-D^{-20}$,
\[
|B\cap E(\mathcal{H}^h_1)|\leq (1+D^{-\eps}h/\ell) m/2^{2h}.
\]
\end{claim}
\begin{proof}[Proof of Claim~\ref{clm:degreedecreasing}]  \textbf{Case 1.} Let $B\in \mathcal{B}$.  Now, by the choice of $\mathcal{H}_2^{h-1}$ (in particular \ref{prop:origdegreebound} or \ref{prop:degreedecreasing}), we always have that $|B\cap E(\mathcal{H}^{h}_0)|)\leq (1+D^{-\eps}(h-1)/\ell) m/2^{2(h-1)}$.
Let $Z_B$ be the set of $z$ such that $\{z\}\cup U_B\in B\cap E(\mathcal{H}^{h}_0)$, so that $|Z_B|\leq (1+D^{-\eps}(h-1)/\ell) m/2^{2(h-1)}$.
For each $z\in Z_B$, let $B_z\in \mathcal{B}_z$ be such that $\{z\}\cup U_B\in B_z$, noting that, therefore, $zB_z\in M^{h-1}\setminus  F_h$.
Note that if $\{z\}\cup U_B\in B\cap E(\mathcal{H}^h_1)$, then we must have that the indicator variable for the cycle/path containing $zB_z$ must be equal to 1 or 0 depending on which matching $zB_z$ is in in the partition $M^{h-1}=M_1^{h-1}\cup\ldots\cup M^{h-1}_{k/2^{h-1}}$, and we must have $\{z\}\cup U_B\in A^h_B$.
Therefore,
\[
\E|B\cap E(\mathcal{H}^{h}_1)|= \frac{1}{4}|Z_B|\leq (1+D^{-\eps}(h-1)/2\ell) m/2^{2h}.
\]

Now, changing whether or not a single edge is in $A^h_B$ or not affects $|B\cap E(\mathcal{H}^{h}_1)|$ by at most 1, and at most $|Z_B|\leq 2m/2^{2(h-1)}=8m/2^{2h}$ of these events can affect $|B\cap E(\mathcal{H}^{h}_1)|$.
For each $i\in [k/2^h]$ and $j\in [r_i^h]$, let $c^h_{i,j}$ be the number of $z\in Z_B$ with $zB_z$ in the path/cycle $S^h_{i,j}$. Then, changing the variable $x_{i,j}^h$ changes $|B\cap E(\mathcal{H}^{h}_1)|$ by at most $c_{i,j}^h\leq s$, while $\sum_{i\in [k/2^h]}\sum_{j\in [r_{i}^h]}c^h_{i,j}\leq |Z_B|$.
Thus
\begin{align*}
\sum_{i\in [k/2^h]}\sum_{j\in [r_{i}^h]}(c_{i,j}^h)^2&\leq s\cdot \sum_{i\in [k/2^h]}\sum_{j\in [r_{i}^h]}c_{i,j}^h
 \leq s\cdot |Z_B|\leq s\cdot 2m/2^{2(h-1)}=8sm/2^{2h}.
\end{align*}

Thus, by Lemma~\ref{lem:mcdiarmid}, $|B\cap E(\mathcal{H}^{h-1}_1)|>(1+D^{-\eps}h/\ell) m/2^{2h}$ with probability at most
\begin{align*}
2\exp\left(-\frac{(\frac{1}{\ell}D^{-\eps}m/2^{2h})^2}{8sm/2^{2h}+8m/2^{2h}}\right)&\leq 2\exp\left(-\frac{(\frac{1}{\ell}D^{-\eps}m/2^{2h})^2}{16sm/2^{2h}}\right)
=2\exp\left(-\frac{D^{-{2}\eps}m}{2^{2h}\cdot \ell^2\cdot 16s}\right)\\
&\leq 2\exp\left(-\frac{D^{1-{
2}\eps}}{16k^3\ell^2s}\right)\leq D^{-20},
\end{align*}
as required.

\medskip

\noindent  \textbf{Case 2.} Let then $B\in \mathcal{B}_v$ with $v\in V$.  Now, by the choice of $\mathcal{H}_2^{h-1}$ (in particular \ref{prop:origdegreebound} or \ref{prop:degreedecreasing}), we always have that $|B\cap E(\mathcal{H}^{h}_0)|\leq (1+D^{-\eps}(h-1)/\ell) m/2^{2(h-1)}$.
Let $E_B=(B\cap E(\mathcal{H}^{h}_0))\cap(\cup_{B'\in \mathcal{B}}B')$, and let $E'_B=(B\cap E(\mathcal{H}^{h}_0))\setminus E_B$.
Note that $|E_B|+|E_B'|\leq (1+D^{-\eps}(h-1)/\ell) m/2^{2(h-1)}$.
For each $e\in E_B$, let $B_e\in \mathcal{B}$ be such that $e\in B_e$, and note that $U_{B_e}\in M^{h-1}$.
For each $e\in E'_B$, let $w_e,z_e$ be such that $e=vw_ez_e$, and let $B^w_e\in \mathcal{B}_{w_e}$ and $B^z_e\in \mathcal{B}_{z_e}$ satisfy $e\in B^w_{e}$ and $e\in B^z_{e}$, and note that $w_eB^w_{e}$ and $z_eB^z_e$ are in $M^{h-1}$ but in different cycles/paths $S^h_{i,j}$ across $i\in [k/2^h]$ and $j\in [r_i^h]$.

For each $e\in E_B$, note that if $e\in B\cap E(\mathcal{H}^h_1)$ then we must have that the indicator variable for the cycle/path containing $B_e$ must be equal to 1 or 0 depending on which matching $U_{B_e}$ is in in the partition $M^{h-1}=M_1^{h-1}\cup\ldots\cup M^{h-1}_{k/2^{h-1}}$ and that $e\in A^h_{B_e}$, which happens with probability $1/4$.
For each $e\in E_B'$ note that if $e\in B\cap E(\mathcal{H}^h_1)$ then we must have that the indicator variables for the cycles/paths containing $w_eB_{w_e}$ and $z_eB_{z_e}$ must be 0 or 1 depending on which matching $w_eB_{w_e}$ and $z_eB_{z_e}$ are in in the partition $M^{h-1}=M_1^{h-1}\cup\ldots\cup M^{h-1}_{k/2^{h-1}}$, which again happens with probability $1/4$.
Therefore,
\[
\E|B\cap E(\mathcal{H}^{h}_1)|= \frac{1}{4}(|E_B|+|E_B'|)\leq (1+D^{-\eps}(h-1)/2\ell) m/2^{2h}.
\]

Now, changing whether or not a single edge $e$ is in $A^h_{B_e}$ or not affects $|B\cap E(\mathcal{H}^{h}_1)|$ by at most 1 and at most $|E_B|\leq 2m/2^{2(h-1)}$ such events affect $|B\cap E(\mathcal{H}^h_1)|$.
For each $i\in [k/2^h]$ and $j\in [r_i^h]$, let $c^h_{i,j}$ be the number of $e\in E_B$ with $U_{B_e}$ in the path/cycle $S^h_{i,j}$ or $e\in E_B'$ with $w_eB_{w_e}$ or $z_eB_{z_e}$ in the path/cycle $S^h_{i,j}$.
Now, for each $wB'\in E(S^h_{i,j})$ with $w\in V\setminus \{v\}$ and $B'\in \mathcal{B}_w$, there are, by \ref{prop:codegreedecreasing}, at most $2D^{2\gamma}m/2^{2(h-1)}k$ edges $e\in B\cap E(\mathcal{H}^h_0)$ with $w_e=w$ and $B_{w_e}=B'$.
Furthermore, for each distinct $e,e'\in E_B\cap E(S)$ we have $B_e\neq B_{e'}$ as $e=\{v\}\cup U_{B_e}$ and $e'=\{v\}\cup U_{B_{e'}}$.
Therefore, as $e(S^h_{i,j})\leq s$, we have $c_{i,j}^h\leq 8sD^{2\gamma}m/2^{2h}k$.
Changing the variable $x_{i,j}^h$ changes $|B\cap E(\mathcal{H}^{h}_1)|$ by at most $c_{i,j}^h$, while
\begin{align*}
\sum_{i\in [k/2^h]}\sum_{j\in [r_{i}^h]}(c_{i,j}^h)^2&\leq (8sD^{2\gamma}m/2^{2h}k)\cdot \sum_{i\in [k/2^h]}\sum_{j\in [r_{i}^h]}c_{i,j}^h\leq (8sD^{2\gamma}m/2^{2h}k)\cdot (|E_B|+2|E_{B'}|)\\
&\leq (8sD^{2\gamma}m/2^{2h}k)\cdot 4m/2^{2(h-1)}
=2^7sD^{2\gamma}m^2/2^{4h}k,
\end{align*}
so $|E_B|+\sum_{i\in [k/2^h]}\sum_{j\in [r_{i}^h]}(c_{i,j}^h)^2\leq 2m/2^{2(h-1)}+2^7sD^{2\gamma}m^2/2^{4h}k\leq 2^8sD^{2\gamma}m^2/2^{4h}k$.
Thus, by Lemma~\ref{lem:mcdiarmid}, $|B\cap E(\mathcal{H}^{h}_1)|<(1+D^{-\eps}h/2\ell) m/2^{2h}$ with probability at most
\[
2\exp\left(-\frac{(\frac{1}{\ell}D^{-\eps}m/2^{2h})^2}{2^8sD^{2\gamma}m^2/2^{4h}k}\right)
=2\exp\left(-\frac{D^{-2\eps-2\gamma}k}{2^8s\ell^2}\right)\leq D^{-20},
\]
where we have used that $k=D^{\mu}$ and $\eps\ll \gamma\ll \mu$.
\renewcommand{\qedsymbol}{$\boxdot$}
\end{proof}
\renewcommand{\qedsymbol}{$\square$}



Putting Claims~\ref{clm:codegreepairsdecreasing} to~\ref{clm:degreedecreasing} together allows us to show that there is likely to be only an edge loss from $\mathcal{H}^h_1$ to $\mathcal{H}^h_2$, as follows.

\begin{claim} \label{clm:minimaledgeloss}
For each $h\in [\ell]$, with probability at least $1-D^{-11}$,
\begin{equation}\label{eqn:minimaledgeloss}
e(\mathcal{H}^h_2)\geq e(\mathcal{H}^h_1)-|V|.
\end{equation}
\end{claim}
\begin{proof}[Proof of Claim~\ref{clm:minimaledgeloss}] Let $h\in [\ell]$. Let $X^h$ be the set of vertices $v\in V$ for which at least one of the following does not hold.
\begin{itemize}
\item There is some $B\in \mathcal{B}_v\cup\{B'\in \mathcal{B}:v\in U_{B'}\}$ such that $|B\cap E(\mathcal{H}^h_2)|\leq (1+h\cdot D^{-\eps}/\ell) m/2^{2h}$.
\item There is some $w\in V\setminus\{v\}$ and $B\in \mathcal{B}_v$, such that the number of edges in $B\cap E(\mathcal{H}^h_2)$ containing $w$ is at most $(1+h/\ell) D^{2\gamma} m/2^{2h}k$.
\item There is some $w\in V\setminus\{v\}$, $B\in \mathcal{B}_v$, and $B'\in \mathcal{B}_w$ with $|(B\cap B')\cap E(\mathcal{H}^h_2)|\leq (1+h/\ell) D^\gamma m/2^{h}k^2$.
\end{itemize}

Note that, as $\Delta(\mathcal{H})\leq D$, using Claims~\ref{clm:codegreepairsdecreasing} to~\ref{clm:degreedecreasing}, we have that, for each $v\in V$, $\P(v\in X^h)\leq 3D^2\cdot D^{-20}$.
Let $E^h$ be the set of edges of $\mathcal{H}$ containing some vertex in $X^h$, so that $\E|E^h|\leq |V|\cdot D\cdot 3D^2\cdot D^{-20}$. Then, by Markov's inequality, we have $|E^h|\geq |V|$ with probability at most $D^{-11}$. Finally, note that if $\mathcal{H}_1^h-E^h$ replaces $\mathcal{H}_2^h$ then \ref{prop:degreedecreasing}--\ref{prop:codegreepairsdecreasing} are satisfied. Thus, by the maximality of $\mathcal{H}_2^h$, $e(\mathcal{H}_2^h)\geq e(\mathcal{H}_1^h)-|E^h|$, and therefore \eqref{eqn:minimaledgeloss} holds with probability at least $1-D^{-11}$.
\renewcommand{\qedsymbol}{$\boxdot$}
\end{proof}
\renewcommand{\qedsymbol}{$\square$}


We now show that there will always be only a small edge loss from $\mathcal{H}^{h-1}_2$ to $\mathcal{H}^h_0$.

\begin{claim} \label{clm:minimaledgeloss2}
For each $h\in [\ell]$,
\[
e(\mathcal{H}^h_0)\geq e(\mathcal{H}^{h-1}_2)-|V|D^{1-\gamma/3}/2^{3h}
\]
\end{claim}
\begin{proof}[Proof of Claim~\ref{clm:minimaledgeloss2}]
Let $h\in [\ell]$. Note that by the minimality in the choice of $F^h$ we have
\[
|F^h|\leq \frac{2|V|}{s}\cdot \frac{k}{2^h}.
\]
Furthermore, for each $B\in \mathcal{B}$ with $U_B\in F^h$, $|B\cap E(\mathcal{H}^{h-1}_2)|\leq 2m/2^{2(h-1)}$ by the choice of $\mathcal{H}^{h-1}_{2}$ (in particular \ref{prop:initialcollectionB} or \ref{prop:degreedecreasing}).
Similarly, for each $v\in V$ and $B\in \mathcal{B}_v$ with $vB\in F^h$, we have that $|B\cap E(\mathcal{H}^{h-1}_2)|\leq 2m/2^{2(h-1)}$ by the choice of $\mathcal{H}^{h-1}_{2}$.
Therefore, the number of edges $e\in E(\mathcal{H}^{h-1}_2)$ for which either $e\in \cup_{B\in \mathcal{B}:U_B\in F^h}B$ or $e\in B$ for some $v\in V$, $B\in \mathcal{B}_v$ and $vB\in F^h$ is at most
\begin{equation}\label{eqn:nottoomany}
|F^h|\cdot 2m/2^{2(h-1)}\leq \frac{2|V|}{s}\cdot \frac{k}{2^h}\cdot \frac{8m}{2^{2h}}\leq \frac{16|V|D}{s2^{3h}}.
\end{equation}

Let $i\in [k/2^h]$ and $j\in [r_i^h]$. For each $B\in \mathcal{B}$, $v\in V$ and $B'\in \mathcal{B}_v$ with $U_B,vB'\in S_{i,j}^h$, we have $B\cap B'=\emptyset$ if $v\in U_B$ and otherwise the only possible edge in $B\cap B'$ is $\{v\}\cup U_B$, so in either case $|B\cap B'|\leq 1$. For each distinct $v,w\in V$ and $B\in \mathcal{B}_v,B'\in \mathcal{B}_w$, we have $|(B\cap B')\cap E(\mathcal{H}^{h-1}_2)|\leq 2D^\gamma m/2^{h-1}k^2$ by the choice of $\mathcal{H}_2^{h-1}$ (in particular \ref{prop:origcodegreepairbound} or \ref{prop:codegreepairsdecreasing}).

Therefore, using the maximality of $\mathcal{H}_h^0$, we have
\begin{align*}
e(\mathcal{H}^{h-1}_2)-e(\mathcal{H}^h_0)&\leq \frac{16|V|D}{s2^{3h}}+\sum_{i\in [k/2^h]}\sum_{j\in [r_i^h]}\frac{e(S_{i,j}^h)^2\cdot 2D^\gamma m}{2^{h-1}k^2}
&&\leq \frac{16|V|D}{s2^{3h}}+\sum_{i\in [k/2^h]}\sum_{j\in [r_i^h]}s\cdot \frac{2e(S_{i,j}^h)\cdot 2D^{1+\gamma}}{2^hk^3}\\
&\leq \frac{16|V|D}{s2^{3h}}+\sum_{i\in [k/2^h]}\frac{8s\cdot |V|\cdot D^{1+\gamma}}{2^hk^3}
&&= \frac{16|V|D}{s2^{3h}}+\frac{8s\cdot |V|\cdot D^{1+\gamma}}{2^{2h}\cdot k^2}\\
&=\frac{|V|D}{2^{3h}}\cdot \left( \frac{16}{s}+\frac{8sD^{\gamma}2^h}{k^2}\right)
&&\leq \frac{|V|D}{2^{3h}}\cdot \left( \frac{16}{s}+\frac{8sD^{\gamma}}{k}\right)
\leq D^{-\gamma/3}\cdot \frac{|V|D}{2^{3h}},
\end{align*}
as required, where we have used that $k=D^\mu$, $s=D^{\gamma/2}$, and $\gamma\ll \mu$.
\renewcommand{\qedsymbol}{$\boxdot$}
\end{proof}
\renewcommand{\qedsymbol}{$\square$}


Now we show that it is likely that the edge loss from $\mathcal{H}^{h}_0$ to $\mathcal{H}^h_1$ is only at most slightly larger than the expected factor of $7/8$.

\begin{claim}\label{clm:edgesdontdecreasetoofast}
For each $h\in [\ell]$, with probability at least $1-D^{-11}$,
\begin{equation}\label{eqn:edgesdontdecreasetoofast}
e(\mathcal{H}_1^h)\geq \frac{1}{8}e(\mathcal{H}_0^h)-D^{1-2\eps}|V|/2^{3h}.
\end{equation}
\end{claim}
\begin{proof}[Proof of Claim~\ref{clm:edgesdontdecreasetoofast}]
Let $E$ be the set of edges $e\in E(\mathcal{H}_0^h)$ for which $e\in \cup_{B\in \mathcal{B}}B$, and let $E'=E(\mathcal{H}_0^h)\setminus E$.
For each $e\in E$, let $B_e\in \mathcal{B}$, $z_e\in V\setminus U_{B_e}$ and $B'_e\in \mathcal{B}_{z_e}$ be such that $e\in B_e\cap B'_e$. Note that $e\in E(\mathcal{H}_1^h)$ exactly if $U_B,z_eB'\in M^h$ and $e\in A^h_{B_e}$, which, by the choice of $\mathcal{H}_0^h$, happens with probability $1/8$.
For each $e\in E'$, label vertices $v_e,w_e,z_e$ and take $B_e^v\in \mathcal{B}_{v_e},B_e^w\in \mathcal{B}_{w_e}$ and $B_e^z\in \mathcal{B}_{z_e}$ such that $e=\{v_e,w_e,z_e\}$ and $e\in B_e^v\cap B_e^w\cap B_e^z$. Note that $e\in E(\mathcal{H}_1^h)$ exactly if $v_eB_e^v,w_eB_e^w,z_eB_e^z\in M^h$, which, by the choice of $\mathcal{H}_0^h$, happens with probability $1/8$.
Therefore, $\E(e(\mathcal{H}_1^h))\geq \frac{1}{8}e(\mathcal{H}_0^h)$.

Now, changing whether or not a single edge is in $A^h_B$ or not affects $e(\mathcal{H}_1^h)$ by at most 1, and the number of these variables is 
(using the choice of $\mathcal{H}_0^h$, and in particular \ref{prop:origdegreebound} or \ref{prop:degreedecreasing}) at most
\[
e(\mathcal{H}_0^h)\leq |V|\cdot (k/2^{h-1})\cdot 2m/2^{2(h-1)}= 16|V|D/2^{3h}.
\]
For each $i\in [k/2^h]$ and $j\in [r_i^h]$, changing the variable $x_{i,j}^h$ changes $e(\mathcal{H}_1^h)$ by  (again using the choice of $\mathcal{H}_0^h$,
and in particular \ref{prop:origdegreebound} or \ref{prop:degreedecreasing})
 at most $c_{i,j}^h:=e(S_{i,j}^h)\cdot 2m/2^{2(h-1)}$.
Note that
\begin{align*}
\sum_{i\in [k/2^h]}\sum_{j\in [r_i^h]}(c_{i,j}^h)^2& \leq \sum_{i\in [k/2^h]}\sum_{j\in [r_i^h]}e(S^h_{i,j})\cdot s\left(\frac{2m}{2^{2(h-1)}}\right)^2\leq \frac{k}{2^h}\cdot 2|V|\cdot s\left(\frac{2m}{2^{2(h-1)}}\right)^2
=\frac{2^7D^2s|V|}{2^{5h}k},
\end{align*}
so that
\[
e(\mathcal{H}_0^h)+\sum_{i\in [k/2^h]}\sum_{j\in [r_i^h]}(c_{i,j}^h)^2\leq\frac{16|V|D}{2^{3h}}+\frac{2^7D^2s|V|}{2^{5h}k}\leq \frac{2^8D^2s|V|}{2^{5h}k}.
\]
Thus, by Lemma~\ref{lem:mcdiarmid}, \eqref{eqn:edgesdontdecreasetoofast} does not hold with probability at most
\begin{align*}
2\exp\left(-\frac{(D^{1-2\eps}|V|/2^{3h})^2}{2^8D^2s|V|/2^{5h}k}\right)
&=2\exp\left(-\frac{D^{-4\eps}k|V|}{2^82^hs}\right)\leq
2\exp\left(-\frac{D^{-4\eps-\gamma/2}|V|}{2^8}\right)\leq D^{-11},
\end{align*}
where we have used that $|V|\geq D^{1/2}$.
\renewcommand{\qedsymbol}{$\boxdot$}
\end{proof}
\renewcommand{\qedsymbol}{$\square$}


Claims~\ref{clm:minimaledgeloss}--\ref{clm:edgesdontdecreasetoofast} allows us to bound the likely edge loss from $\mathcal{H}_1$ to $\mathcal{H}'$. Indeed, using Claims~\ref{clm:minimaledgeloss}--\ref{clm:edgesdontdecreasetoofast}, we have, with probability at least $1-3\ell D^{-11}\geq 1-D^{-10}$, that, for each $h\in [\ell]$, $e(\mathcal{H}^h_0)\geq e(\mathcal{H}^{h-1}_2)-D^{1-\gamma/3}\cdot |V|/2^{3h}$,
$e(\mathcal{H}_1^h)\geq \frac{1}{8}e(\mathcal{H}_0^h)-D^{1-2\eps}|V|/2^{3h}$ and
$e(\mathcal{H}^h_2)\geq e(\mathcal{H}^h_1)-|V|$. When this holds, we have, for each $h\in [\ell]$, that
\[
e(\mathcal{H}^h_2)\geq \frac{1}{8}e(\mathcal{H}^{h-1}_2)- 3D^{1-2\eps}|V|/2^{3h},
\]
and therefore, it is easy to show by induction that, for each $h\in [\ell]$, $e(\mathcal{H}^h_2)\geq 2^{-3h}\cdot e(\mathcal{H}^{h-1}_0)-h\cdot 3D^{1-2\eps}|V|/2^{3h}$. Thus, we have
\[
e(\mathcal{H}')=e(\mathcal{H}^\ell_2)\geq 2^{-3\ell}\cdot e(\mathcal{H}^{0}_2)-\ell\cdot 3D^{1-2\eps}|V|/2^{3\ell}\geq \frac{1}{t}e(\mathcal{H}^{0}_2)-\frac{1}{2t}D^{1-\eps}|V|.
\]
To complete the proof of \ref{prop:manyedges}, it is sufficient then to show that $e(\mathcal{H}_1)=e(\mathcal{H}^{0}_2)\geq e(\mathcal{H})-D^{1-\eps}|V|/2$.


For this, first note that, from \ref{prop:removedoublecodegreeedges}, for each $v\in V$, there are at most $D/m\leq 2D^\mu$ vertices $x\in V$ with $\mathrm{cod}_{\mathcal{H}}(v,x)\geq m$, and, thus, $e(\mathcal{H})-e(\mathcal{H}_0)\leq |V|\cdot (2D^\mu)^2$.
Take an arbitrary $\mathcal{B}$ satisfying \ref{prop:initialcollectionB}. For each $v\in V$, let
\[
k_v=\lfloor|\{e\in E(\mathcal{H}_0)\setminus \cup_{B\in \mathcal{B}:v\in U_B}B:v\in V(e)\}|/m\rfloor
\]
and, if $k_v\geq k\cdot D^{-\gamma/4}$, then uniformly at random take a collection $\mathcal{B}_v$ of $k_v$ disjoint subsets of size $m$ in $\{e\in E(\mathcal{H}_0)\setminus \cup_{B\in \mathcal{B}:v\in U_B}B:v\in V(e)\}$, and, otherwise, take $\mathcal{B}_v=\emptyset$.

Now, letting $Z\subseteq E(\mathcal{H}_0)$ to be the set of edges $e\in E(\mathcal{H}_0)$ for which either there is some $B\in \mathcal{B}$ with $e\in B$ and some $v\in V$ and $B'\in \mathcal{B}_v$ with $e\in B'$ or, for each $v\in V(e)$, there is some $B\in \mathcal{B}_v$ with $e\in B$, we have
\[
|E(\mathcal{H}_0)\setminus Z|\leq
m\cdot (k\cdot D^{-\gamma/4}+1)\cdot |V|\leq D^{1-\eps}|V|/12.
\]
Furthermore, for each distinct $v,w\in V$ with $k_v\geq k\cdot D^{-\gamma/4}$, as there are at most $m$ edges in $E(\mathcal{H}_0)\setminus \cup_{B\in \mathcal{B}:v\in U_B}B$ containing both $v$ and $w$ by the maximality of $\mathcal{B}$, the probability that there is some $B\in \mathcal{B}_v$ for which there are at least $2m/k_v\leq D^{2\gamma}m/k$ edges in $B$ containing $w$ is, using an appropriate Chernoff bound for hypergeometric random variables, at most $k\cdot D^{-10}$.
Similarly, for each distinct $v,w\in V$ with $k_v,k_w\geq k\cdot D^{-\gamma/4}$, the probability that there is some $B\in \mathcal{B}_v$ and $B'\in \mathcal{B}_w$ for which $|B\cap B'|\leq 2m/k_vk_w\leq D^{\gamma}n/k^2$ is at most $k^2\cdot D^{-10}$.

Therefore, with positive probability,  if we delete, for each $v\in V$, each $B\in \mathcal{B}_v$
for which \ref{prop:origcodegreebound} or \ref{prop:origcodegreepairbound} does not hold, then this deletes certainly at most $D^{1-\eps}|V|/12$ edges.
Thus, by the maximality of the $\mathcal{B}_v$, $v\in V$, we have that 
$Z$ contains all but at most $D^{1-\eps}|V|/6$ edges of $e(\mathcal{H}_0)$.
Note that, by the definition of $\mathcal{H}_1$, we therefore have that $\mathcal{H}_1$ contains all but at most $D^{1-\eps}|V|/2$ of the edges of $e(\mathcal{H}_0)$. Thus, we have that $e(\mathcal{H}_1)\geq e(\mathcal{H})-(2D^\mu)^2|V|-D^{1-\eps}|V|/2\geq e(\mathcal{H})-D^{1-\eps}|V|$, as required.
\end{proof}


\subsection{Proof of Theorem~\ref{thm:semiunifieddecomp}}\label{sec:deduce}
We now use Lemma~\ref{lem:semiunifieddecomp} and Corollary~\ref{cor:molloyreed} to prove Theorem~\ref{thm:semiunifieddecomp}.
\begin{proof}[Proof of Theorem~\ref{thm:semiunifieddecomp}] Let $\eps,\mu,\eta>0$ and $D_0$ have the property from Lemma~\ref{lem:semiunifieddecomp}, as shown to exist by that lemma, let $\xi\ll \eps,1/k$, and, moreover, assume that $1/D_0\ll \eps,1/k$.
We will show that Theorem~\ref{thm:semiunifieddecomp} holds for $\xi,\eta$ and $D_0$. Let then $D\geq D_0$ and let $\mathcal{H}$ be a $3$-uniform hypergraph with maximum degree $D$ and at least $\sqrt{D}$ vertices for which the graph on $V(\mathcal{H})$
with edges $xy$ present if $\mathrm{cod}_{\mathcal{H}}(x,y)\geq D^{1-\eta}$ is bipartite.

Using the property from Lemma~\ref{lem:semiunifieddecomp}, there is some $D^{3\mu}\leq t\leq 8D^{3\mu}$ for which, setting $r=t^2$, we can find independent random subhypergraphs $\mathcal{H}'_1,\ldots,\mathcal{H}'_r$ of $\mathcal{H}$ such that, for each $i\in [r]$,
\begin{itemize}
\item for each $e\in E(\mathcal{H})$, $\P(e\in E(\mathcal{H}'_i))\leq 1/t$,
\item $\Delta(\mathcal{H}'_i)\leq (1+D^{-\eps})D/t$,
\item for each distinct $x,y\in V(\mathcal{H}_i')$, either $\mathrm{cod}_{\mathcal{H}'_i}(x,y)\leq D^{1-\eps}/t$ or every $e\in E(\mathcal{H}')$ satisfies $\{x,y\}\subset V(e)$ or $\{x,y\}\cap V(e)=\emptyset$, and,
\item with probability at least $1-D^{-10}$, $e(\mathcal{H}'_i)\geq e(\mathcal{H})/t-D^{1-\eps}|\mathcal{H}|/t$.
\end{itemize}
By a simple application of a Chernoff bound, we have, for each $e\in E(\mathcal{H})$, that
\[
\P\left(|\{i\in [r]:e\in E(\mathcal{H}_i')\}|>(1+D^{-\eps})\frac{r}{t}\right)<D^{-10}.
\]
Therefore, by an application of Markov's lemma, with probability $1-D^{-5}$, for all but at most $|\mathcal{H}|$ edges $e\in E(\mathcal{H})$ we have $|\{i\in [r]:e\in E(\mathcal{H}_i')\}|\leq (1+D^{-\eps})\frac{r}{t}$.
Let $\mathcal{H}_1$ be the subhypergraph of $\mathcal{H}$ of these edges.
With positive probability then, we can assume that $e(\mathcal{H}_1)\geq e(\mathcal{H})-|\mathcal{H}|$, and that $e(\mathcal{H}'_i)\geq e(\mathcal{H})/t-D^{1-\eps}|\mathcal{H}|/t$ for each $i\in [r]$.

Let $m=(1+D^{-\eps})r/t$.
For each $e\in E(\mathcal{H}_1)$, independently at random pick $i_e\in \{0\}\cup \{i\in [r]:e\in E(\mathcal{H}_i')\}$ so that $\P(i_e=i)=1/m$ for each $i\in [r]$ for which $e\in E(\mathcal{H}_i')$.
For each $i\in [r]$, let $\mathcal{H}''_i$ be the subhypergraph of $\mathcal{H}'_i$ of edges $e\in E(\mathcal{H}_i')$ with $i_e=i$.
Using a Chernoff bound and the local lemma, we can assume that, for each $i\in [r]$, the following hold
\begin{itemize}
\item $\Delta(\mathcal{H}''_i)\leq (1+D^{-\eps/2})D/tm$,
\item for each distinct $x,y\in V(\mathcal{H}_i'')$, either $\mathrm{cod}_{\mathcal{H}''_i}(x,y)\leq D^{1-\eps/2}/tm$ or every $e\in E(\mathcal{H}''_i)$ satisfies $\{x,y\}\subset V(e)$ or $\{x,y\}\cap V(e)=\emptyset$, and,
\item $e(\mathcal{H}''_i)\geq e(\mathcal{H})/tm-2D^{1-\eps}|\mathcal{H}|/tm$.
\end{itemize}
Now, as $\xi,1/D_0\ll \eps,1/k$, by Corollary~\ref{cor:molloyreed}, $\chi'(\mathcal{H}''_i)\leq (1+D^{-\xi})D/tm$ for each $i\in [r]$.
Let $\mathcal{H}'=\cup_{i\in [r]}\mathcal{H}''_i$, so that
\[
\chi'(\mathcal{H}')\leq r\cdot (1+D^{-\xi})D/tm\leq  (1+D^{-\xi})D,
\]
and
\[
e(\mathcal{H}')\geq r\cdot e(\mathcal{H})/tm-r\cdot 2D^{1-\eps}|\mathcal{H}|/tm
= (1+D^{-\eps})^{-1}\cdot e(\mathcal{H})-(1+D^{-\eps})^{-1}\cdot 2D^{1-\eps}|\mathcal{H}|
\geq e(\mathcal{H})-|\mathcal{H}|\cdot D^{1-\xi},
\]
and hence $\mathcal{H}'$ satisfies our requirements.
\end{proof}


\subsection{Proof of Theorem~\ref{thm:lowerbounddecomp}}
We now deduce Theorem~\ref{thm:lowerbounddecomp} from Theorem~\ref{thm:semiunifieddecomp}.
\begin{proof}[Proof of Theorem~\ref{thm:lowerbounddecomp}] Let $\xi,\eta>0$ and $D_0$ be such that the property in Theorem~\ref{thm:semiunifieddecomp} holds for each $D\geq D_0$ with $\xi$ and $\eta$, and let $0<\eps\leq \xi/3$.  Assume moreover that $\eps>0$ is small enough that
$n-n^{1-\eps} \leq 1$ for each $n<D_0$. We will show that Theorem~\ref{thm:lowerbounddecomp} holds for $\eps$. Let $S$ then be an equi-$n$-square. Note that if $n<D_0$, then all that is required is for $S$ to contain 1 transversal with at least 1 cell, which trivially holds. Assume then that $n\geq D_0$. By making $\eps$ even smaller if necessary, we may also assume $n^{\xi/3}> 2$.

Say the set of rows of $S$ is $\mathcal{I}$, the set of columns is $\mathcal{J}$ and the set of symbols is $\mathcal{A}$. Form a 3-uniform hypergraph $\mathcal{H}$ on $\mathcal{I}\cup \mathcal{J}\cup \mathcal{A}$ by adding, for each $i\in \mathcal{I}$ and $j\in \mathcal{J}$, an edge $ijc$ where $c\in \mathcal{A}$ is the symbol in the square of $S$ indexed by $(i,j)$.
If $\mathcal{H}$ contains at least $n-n^{1-\eps}$ disjoint matchings with at least $n-n^{1-\eps}$ edges, then it can be easily seen that $S$ has at least $n-n^{1-\eps}$ disjoint transversals with size at least $n-n^{1-\eps}$. Suppose then that $\mathcal{H}$ does not contain at least $n-n^{1-\eps}$ disjoint matchings with at least $n-n^{1-\eps}$ edges.

Note that if $i\in \mathcal{I}$ and $j\in \mathcal{J}$ then $\mathrm{cod}_{\mathcal{H}}(i,j)=1$. Therefore, the graph on $V(\mathcal{H})$ with edges $xy$ present if $\mathrm{cod}_{\mathcal{H}}(x,y)\geq n^{1-\eta}$ is bipartite. 
Thus, by Theorem~\ref{thm:semiunifieddecomp}, there is some subhypergraph $\mathcal{H}'\subset \mathcal{H}$ with $\chi'(\mathcal{H})\leq (1+n^{-\xi})n$ and
$e(\mathcal{H}')\geq e(\mathcal{H})-n^{2-\xi}$. 
However, we then have
\begin{align*}
e(\mathcal{H}')&\leq (n-n^{1-\eps})n+(\chi'(\mathcal{H}')-n+n^{1-\eps})(n-n^{1-\eps})\\
&\leq (n-n^{1-\eps})n+(n^{1-\xi}+n^{1-\eps})(n-n^{1-\eps})
\leq n^2+n^{2-\xi}-n^{2-2\eps},
\end{align*}
so that $e(\mathcal{H})\leq e(\mathcal{H}')+n^{2-\xi}\leq n^2+2n^{2-\xi}-n^{2-2\eps}<n^2$, as $\eps\le \xi/3$ and we assumed $n^{\xi/3}>2$, a contradiction as $e(\mathcal{H})=n^2$. Thus, $S$ has at least $n-n^{1-\eps}$ disjoint transversals with size at least $n-n^{1-\eps}$, as required.
\end{proof}


\bibliographystyle{abbrv}
\bibliography{equinsquares}

\appendix

\section{Almost-full cyclic transversals}

A strengthening of Stein's conjecture, originally stated in the language of sub-Ramsey numbers, was also disproved by Pokrovskiy and Sudakov~\cite{pokrovskiy2019counterexample}. In particular, Hahn~\cite{hahn1986path} had conjectured that in every edge-colouring of $K_n$ with at most $n/2-1$ edges of each colour, one can find a rainbow path on $n$ vertices. Similarly to Stein's conjecture, we show that Hahn's conjecture is also asymptotically true. In fact, we prove a slightly stronger statement for edge coloured complete digraphs, which, as discussed below, has an equivalent formulation for $n\times n$ arrays. From these two equivalent results, we deduce two corollaries, one of them being an asymptotic version of Hahn's conjecture and the other about \textit{cyclic transversals} in equi-$n$-squares. 

There is a one-to-one correspondence between edge colourings of $\stackrel{\leftrightarrow}{K_n}$ and $n \times n$ arrays filled with symbols. Indeed, let $\stackrel{\leftrightarrow}{K_n}$ denote the complete digraph on $n$ vertices with all $n^2$ directed edges and suppose that $\chi :E(\stackrel{\leftrightarrow}{K_n}) \rightarrow C$ is an edge colouring. We construct an $n \times n$ array $S$ by putting the symbol $\chi(ij)$ in the $(i,j)$th cell of $S$. If we instead consider edge colourings of the complete graph, $K_n$, then this would restrict to all symmetric $n \times n$ arrays (with no symbols on the leading diagonal, say).

\begin{theorem}\label{thm:coloureddigraph}
    There exists an $\alpha>0$ so that the following holds. Let $G$ be the complete digraph on $n$ vertices, $\stackrel{\leftrightarrow} {K_n}$, which is edge coloured so that each colour appears at most $n$ times. Then $G$ admits a rainbow directed cycle with length at least $n-n^{1-\alpha}$. 
\end{theorem}

As discussed above, this is equivalent to the following theorem concerning $n \times n$ arrays. We call a transversal $T$ of an $n \times n $ array a \textit{cyclic transversal} if for some $k \leq n$, $T$ has the form
\[T=\{(a_1,a_2),(a_2,a_3),\dots,(a_{k-1},a_k),(a_k,a_1)\}.
\]
Moreover, a full cyclic transversal is also referred to as a Hamilton transversal in the literature.

\begin{theorem}[Reformulation of Theorem~\ref{thm:coloureddigraph}]\label{thm:cyclictransversalinarray}
    There exists an $\alpha>0$ so that the following holds. Let $S$ be an $n \times n $ array filled with symbols so that each symbol has at most $n$ appearances. Then $S$ admits a cyclic transversal with size at least $n-n^{1-\alpha}$.
\end{theorem}

From the above (equivalent) theorems we deduce the following corollaries, respectively.

\begin{corollary}[Asymptotic Hahn, Corollary to Theorem~\ref{thm:coloureddigraph}]\label{cor:asymptotichahn}
   There exists an $\alpha>0$ so that the following holds. Let $G$ be the complete graph on $n$ vertices which is edge-coloured so that each colour has at most $n/2$ appearances. Then $G$ admits a rainbow cycle with size at least $n-n^{1-\alpha}$.
\end{corollary}

\begin{corollary}[Corollary to Theorem~\ref{thm:cyclictransversalinarray}]\label{cor:cyclictransversalinequinsquare}
    There exists an $\alpha>0$ so that the following holds. Let $S$ be an equi-$n$-square. Then $S$ contains a cyclic transversal with size at least $n-n^{1-\alpha}$.
\end{corollary}

We show above that every equi-$n$-square can be almost decomposed into almost full transversal and further that every equi-$n$-square admits an almost full cyclic transversal. It is therefore natural to ask whether every equi-$n$-square can be almost decomposed into almost full cyclic transversal.
\begin{conjecture}
    Every equi-$n$-square contains $(1-o(1))n$ disjoint cyclic transversals, each with size $(1-o(1))n$.
\end{conjecture}

\section{Proof of Theorem~\ref{thm:coloureddigraph}}\label{sec:digraphproof}

In this section, we prove Theorem~\ref{thm:coloureddigraph}. In Section~\ref{sec:directedoutline}, we give an outline of the proof, before recalling some more tools we need in Section~\ref{sec:directedtools}. 
Section~\ref{subsection:noshortcycles}, which may be of independent interest, is dedicated to showing how to apply the semi-random method to an edge-coloured digraph, so as to obtain an almost spanning rainbow directed path forest with no short paths.Section~\ref{subsection:AHproof} contains the proof of~\cref{thm:coloureddigraph}. 
We use standard notation for directed graphs/digraphs, which we now highlight. For a digraph $G$ and vertex $v\in V(G)$ we let $N^+_G(v)=\{u\in V(G) :\stackrel{\longrightarrow}{vu}\in E(G)\}$ denote the out-neighbourhood of $v$ in $G$, also called the ($+$)-neighbourhood of $v$ in $G$. Similarly, we let $N^-_G(v)=\{u\in V(G) : \stackrel{\longrightarrow}{uv}\in E(G)\}$ denote the in-neighbourhood of $v$ in $G$, also called the ($-$)-neighbourhood of $v$ in $G$. We let $d^+_G(v)=|N^+_G(v)|$ and $d^-_G(v)=|N^-_G(v)|$ denote the in- and out-degree of $v$, respectively. We let $\Delta^+(G),~\Delta^-(G),~\delta^+(G)$ and $\delta^-(G)$ denote the maximum out-degree, maximum in-degree, minimum out-degree and minimum in-degree, respectively and set $\Delta(G)=\text{max}\{\Delta^+(G),\Delta^-(G)\}$ and $\delta(G)=\text{min}\{\delta^+(G),\delta^-(G)\}$.

\subsection{Sketch proof of Theorem~\ref{thm:coloureddigraph}}\label{sec:directedoutline}
We begin with an edge coloured complete digraph $G$, in which each colour appears at most $n$ times. In order to have a better grasp on the number of colours, we combine colour classes if they are too small, so that we have at most $3n$ distinct colours, say. We may view $G$ as a $3$-uniform, $3$-partite hypergraph instead. Indeed, let $H$ be the hypergraph on vertex classes $V_1,V_2$ and $C$, where $V_1$ and $V_2$ are both copies of $[n]$, the vertex set of $G$, to which we add the edge $uvc$ whenever $\chi({\stackrel{\longrightarrow}{uv}})=c$. To this hypergraph $H$, we apply our main result,~\cref{lem:semiunifieddecomp}, which gives us a random subhypergraph with an all-or-small codegree condition. We go back to digraphs and think of it as $\mathcal{D}$, an $n$-vertex digraph satisfying 
\begin{enumerate}
    \item $\Delta(\mathcal{D}) \leq (1+n^{-\varepsilon})n^{1-3\mu}$,
        \item $e(\mathcal{D})\geq (1-5n^{-\varepsilon})n^{2-3\mu}$,
        \item each colour $c \in C$ appears at most $(1+n^{-\varepsilon})n^{1-3\mu}$ times,
        \item \label{prop:digraphallorsmall}for any vertex $v \in V(\mathcal{D)}$ and colour $c \in C$, either $c$ appears at most $n^{1-3\mu-\varepsilon}$ times among the in/(out)-edges of $v$, or each in/(out)-edge of $v$ has colour $c$, in which case $c$ does not appear elsewhere.
\end{enumerate}
Applying the semi-random method, for example as~\cref{cor:molloyreed}, at this point would give us a rainbow collection of directed paths and cycles with size $(1-o(1))n$. Unsurprisingly then, to obtain a single long rainbow directed cycle, we set aside a set of vertices $R$ and colours $C_1 \cup C_2$. We consider the all-or-small codegree condition between vertices and colours when reserving these vertices and colours. Informally, if a vertex has high codegree with a colour, it makes sense to reserve that vertex if and only if we reserve the colour, as otherwise one of them would be wasted. Next, in what is left, $\mathcal{D'}$, we find a rainbow directed linear forest with size $(1-o(1))n$ which has no short paths, a condition we require for technical reasons. This is done in Section~\ref{subsection:noshortcycles} which may be of independent interest. We connect these rainbow directed paths $P_1,P_2,\dots$ by using the reserved vertices $R$ and colours $C_1 \cup C_2$. In more detail, we connect vertices $u,v$ by paths with length three, with internal vertices in $R$ whose first and third edges are coloured by $C_2$ and second edge by $C_1$. Separating the reserved colours into two sets helps break down the proof into smaller chunks and essentially allows us to worry about one less issue at a time. It is easy to make such a connection between $u$ and $v$, but we would soon run into trouble if we tried to carelessly do it one connection at a time. We complete all the connections simultaneously by constructing auxiliary $5$-uniform hypergraphs, so that an edge in this hypergraph corresponds to a suitable path with length three which can be used to connect $u$ to $v$. A matching $e_1,e_2,\dots$ in the union of these auxiliary hypergraphs, for which $e_i \in H_i$, then would tell us how to connect all the paths without any vertex/colour obstructions, i.e. how to construct a single directed rainbow cycle. We prove the existence of such a matching by using a version of Hall's matching theorem for Hypergraphs by Aharoni and Haxell~\cite{aharoni2000hall}, which completes the proof.

\subsection{Preliminaries}\label{sec:directedtools}
The tools we will need are our main result, Lemma~\ref{lem:semiunifieddecomp}, which finds a random subhypergraph with all-or-small codegrees, probabilistic inequalities like Chernoff's inequality and McDiarmid's inequality, Lemma~\ref{lem:mcdiarmid}, a version of Hall's matching theorem for hypergraphs~\cite{aharoni2000hall} and a version of the nibble, such as Molloy and Reed~\cite{molloy2000near}.

\begin{theorem}\label{thm:molloyreed}
For all $k$ there is a constant $C_k$ such that any $k$-uniform hypergraph of maximum codegree $B$ and maximum degree $\Delta$ has list chromatic index at most $\left(1+C_k(B/\Delta)^{1/k}(\log (\Delta/B))^4\right)\Delta$.
\end{theorem}

\begin{lemma} Let $n\in \N$ and $c_1,\ldots,c_n\geq 0$. For each $i\in [n]$, let $X_i$ be an independent random variable taking values in $\Omega_i$, and let $X=(X_1,\ldots,X_n)$.
Let $f:\prod_{i=1}^n\Omega_i\to \mathbb{R}$ be a function such that, for each $i\in [n]$, changing $X$ in the $i$th co-ordinate changes the value of $f(X)$ by at most $c_i$.

Then, for all $t>0$,
\[
\P(|f(X)-\mathbb{E}(f(X))|>t)\leq 2\exp\left(-\frac{t^2}{\sum_{i\in [n]}c_i^2}\right).
\]

\end{lemma}

In the below, if $\mathcal{A}=\{H_1,\cdots,H_t\}$ is a collection of $k$-uniform hypergraphs, then a system of disjoint representatives for $\mathcal{A}$ is a matching $\{e_1,\cdots,e_t\}$ so that for each $i \in [t]$, $e_i \in H_i$.

\begin{theorem}[Hypergraph Hall matching theorem~\cite{aharoni2000hall}]\label{thm:aharonihaxell}
Let $\mathcal{A}$ be a family of $k$-uniform hypergraphs. A sufficient condition for the existence of a system of disjoint representatives for $\mathcal{A}$ is that for every $\mathcal{B}\subseteq \mathcal{A}$ there exists a matching in $  \bigcup_{H\in \mathcal{B}}H$ with size greater than $k(|\mathcal{B}|-1)$.
\end{theorem}

\subsection{Finding a large rainbow linear forest with no short paths} \label{subsection:noshortcycles}
For a digraph $G$, we denote by $\Delta(G)$ the maximum of the maximum out-degree and the maximum in-degree.
\begin{theorem}\label{lem:noshortcycles}
Let $0 < 1/n \ll \gamma \ll \eps \ll 1$. Let $G$ be an edge-colored digraph with the set of colors $C$. Let $D \ge \log^2 n$. Suppose the following hold. 
\stepcounter{propcounter}
\begin{enumerate}[label = {\textup{\textbf{\Alph{propcounter}\arabic{enumi}}}}]
\item $\Delta(G) \le D$.
\item $e(G) \ge (1-D^{-\eps}) D n$.
\item Every color appears at most $D$ times in $G$.
\item \label{noshortcycle:allorsmall}For every color $c\in C$ and every vertex $v\in V(G)$ and orientation $\sigma \in \{+,-\}$, either $d_c^\sigma(v) \le D^{1-\eps}$ or all $\sigma$-oriented edges from $v$ are coloured with $c$ and the colour $c$ does not appear elsewhere.
\end{enumerate}
Then, $G$ contains a rainbow directed path forest containing at least $(1-D^{-\gamma})n$ vertices such that each path has length between $\ell$ and $2\ell$, where $\ell= D^{\gamma}$. 
\end{theorem}

\begin{proof}
We may assume, by combining different colours together, that each colour appears at least $D/4$ times and so in particular that there are at most $4n$ distinct colours. Call a triple $(v,\sigma,c) \in V(G) \times \{+,-\}\times C$ a \textit{high codegree triple} if $c$ appears on more than $D^{1-\varepsilon}$ $\sigma$-oriented edges from $v$ in $G$, in which case it follows by~\ref{noshortcycle:allorsmall} that the colour $c$ appears on all $\sigma$-oriented edges from $v$ in $G$ and on no other edges. For such a high codegree triple $(v,\sigma,c)$ we define $hc(v,\sigma):=c$ to be the \textit{high codegree colour}, $hp(c):=(v,\sigma)
        $ to be the \textit{high codegree pair} and $hv(c):=v$ to be the \textit{high codegree vertex}. Note that each vertex $v \in V(G)$ could have up to two different high codegree colours, hence specifying the orientation $\sigma \in \{+,-\}$ is important. Let $C_h$ denote the set of all high codegree colours $c \in C$ and set $C_\ell =C \setminus C_h$ to be the set of \textit{low codegree colours}, i.e. all colours $c\in C$ which appear on at most $D^{1-\varepsilon}$ $\sigma$-oriantated edges from $v$ in $G$ for each $(v,\sigma)\in V(G)\times \{+,-\}$.

We randomly partition $V(G)=V_1\cup V_2 \cup \dots \cup V_k$ into $k = D^{5\gamma}=\ell^5$ disjoint sets, by independently assigning each vertex to one of the sets with equal probability. Subsequently, we partition the set of high codegree colours as $C_h=C_{h,1}\cup C_{h,2}\cup\dots \cup C_{h,k}$, where for $i \in [k]$ we set
\begin{equation} \label{equ:highcodegreecolourclass}
C_{h,i}:=\{c\in C_h:hp(c)=(v,\sigma)~\text{and}~ \sigma=+,v\in V_i~\text{or}~\sigma=-,v\in V_{i+1}\}.
\end{equation}
For the set of low codegree colours, randomly partition it as $C_\ell =C_{\ell,1}\cup C_{\ell,2}\cup\dots\cup C_{\ell,k}$ by independently assigning each colour to one of the sets with equal probability. This induces a partition of the colour set $C=C_1\cup C_2\cup \dots\cup C_k$, where $C_i=C_{h,i}\cup C_{\ell,i}$, in which each colour $c \in C$ is equally like to be assigned to any of the sets. For a high codegree triple $(v,\sigma,c)$, let $s(v,\sigma,c)$ be a random subset of the $\sigma$-oriented edges from $v$ in $G$, in which each edge is taken independently at random with probability $1-D^{-5\gamma}$. Let $S=\bigcup s(v,\sigma,c)$ be the collection of all such edges.

 \begin{claim} \label{claim:rainbowforestnoshortpaths}
      There exists a choice of the sets $V_1,V_2,\dots,V_k$ and $C_{\ell,1},C_{\ell,2},\dots,C_{\ell,k}$, so that the following properties hold:
\begin{enumerate}
\item For every $i\in [k]$, the number of directed edges $\stackrel{\longrightarrow}{uv}$ such that $u\in V_i$, $v\in V_{i+1}$, $\stackrel{\longrightarrow}{uv} \in E(G)\setminus S$ and the colour of $\stackrel{\longrightarrow}{uv}$ is in $C_i$ is at least $(1-2D^{-\eps})Dn/k^{3}$.

\item For every $v\in V(G)$ and $i\in [k]$ and $\sigma\in \{+,-\}$, there are at most $D/k^2+D^{2/3}$ $\sigma$-oriented edges from $v$ to $V_i$ in $E(G) \setminus S$ whose colour is in $C_i$.

\item For every $i\in [k]$ and $c\in C$, the number of $c$-coloured directed edges in $\{\stackrel{\longrightarrow}{uv}\in E(G) \setminus S : u\in V_i, v\in V_{i+1}\}$ is at most $D/k^{2} + D^{2/3}$.
\end{enumerate}
\end{claim}
\claimproofstart[Proof of Claim~\ref{claim:rainbowforestnoshortpaths}]
    The second and third conditions can easily be deduced from Chernoff's inequality, so we only discuss the first. For each vertex or low codegree colour $j^\star \in V(G) \cup C_\ell$, let $X_{j^\star}$ be the random variable which indicates which vertex/colour class $j^\star$ belongs to, among $V_1,V_2,\dots V_k$ or $C_{\ell,1},C_{\ell,2},\dots C_{\ell,k}$, respectively. Then, $X_1,X_2,\dots,X_m$ are independent random variables and there are at most $m \leq 5n$ of them, as there are $n$ vertices and at most $4n$ colours. Furthermore, $X=(X_1,X_2,\dots,X_m)$ fully determines both the vertex and colour classes, as a high codegree colour is assigned a colour class subject to which vertex class its high codegree vertex was assigned~(\ref{equ:highcodegreecolourclass}). For $i \in [k]$, consider $f_i(X)$ to be the number of directed edges $\stackrel{\longrightarrow}{uv}$ so that $u \in V_i$, $v \in V_{i+1}$, $\stackrel{\longrightarrow}{uv}\in E(G) \setminus S$ and the colour of $\stackrel{\longrightarrow}{uv}$ is in $C_i$. As each edge of $G$ satisfies the requirements with probability $1/k^3$, we see that $\mathbb{E}f_i(X)\geq (1-D^{-\varepsilon})Dn/k^3$. Further, changing the value of $X_{j^\star}$ changes $f_i(X)$ by at most $c_{j^\star}\leq D$. Then we have that $\sum_{j^\star\leq m}c_{j^\star}^2\leq 5nD^2$. Applying~\cref{lem:mcdiarmid} with $t=D^{1-2\varepsilon}n/k^3$ gives
    \[
    \mathbb{P}(|f_i(X)-\mathbb{E}f_i(X)|>D^{1-2\varepsilon}n/k^3)\leq 2 \exp\left(-\frac{nD^{-4\varepsilon-30\gamma}}{5}\right)=\exp(-n^{\Omega(1)}).
    \]
It follow that $f_i(X)\geq (1-2D^{-\varepsilon})nD/k^3$ with probability $1-\exp(-n^{\Omega(1)})$. Taking a union bound over all $i \in [k]$, we can find a partition of the vertices and colours so that all three conditions hold.\claimproofend

 For $i \in [k]$, let $G_i$ be the bipartite digraph on vertex classes $V_i$ and $V_{i+1}$, whose edge set is all $\stackrel{\longrightarrow}{uv} \in (V_i\times V_{i+1})\cap(E(G)\setminus S)$ such that $\stackrel{\longrightarrow}{uv}$ is a directed edge with colour in $C_i$. Our goal is to find a large rainbow matching in each $G_i$ using \Cref{cor:molloyreed}. For that, consider the $3$-uniform $3$-partite hypergraph $H_i$ on the vertex set $(V_i,V_{i+1},C_i)$ where the edge $uvc$ is in $H_i$ iff $\stackrel{\longrightarrow}{uv}\in E(G_i)$ and the colour of $\stackrel{\longrightarrow}{uv}$ is $c$. From the three properties above and hypothesis in the statement, we deduce the following. 
\begin{enumerate}
\item $e(H_i) \ge (1-2D^{-\eps})Dn/k^{3}$.
\item $\Delta(H_i) \le D/k^{2} + D^{2/3}$ .
\item For each distinct $x,y\in V(H_i)$, either $\mathrm{cod}_{H_i}(x,y)\leq D^{1-\eps}$ or every edge contains either both $x$ and $y$ or none of them.
\end{enumerate}
Using these and invoking \Cref{cor:molloyreed} to $H_i$, we find a matching in $H_i$ of size at least 
\[
\frac{e(H_i)}{\chi'(H_i)}
\ge \frac{(1-2D^{-\eps})Dn/k^{3}}{(1+D^{-5\gamma}) (D/k^{2} + D^{2/3})} 
\ge (1-2D^{-5\gamma}) n/k.
\]
Consider such a matching $E_i$ for every $i\in [k]$.
If we consider the set of all directed edges in $G$ corresponding to the edges in $E_i$ for $i\in [k]$, we find a rainbow directed graph $F$ with at least $(1-2D^{-5\gamma}) n$ edges such that the connected components of $F$ are either directed cycles of length at least $k$ or directed paths of any length, and the total number of connected components is at most $2D^{-5\gamma} n$. Now, we delete a minimal set of edges from $F$ to obtain $F'$ whose connected components are directed paths of length between $\ell$ and $2\ell$. This deletes at most 

\[n/k +n/(2\ell) + \ell \cdot 2D^{-5\gamma} n =D^{-5\gamma}n+\frac{D^{-\gamma}}{2}n+2D^{-4\gamma}n\le \frac{3}{4}D^{-\gamma} n\]

edges. Indeed, at most $n/k$ edges need be deleted to get rid of cycles, $n/(2\ell)$ to break longer paths into paths with length between $\ell$ and $2\ell $ and finally, at most $\ell$ edges need to be deleted from each small component and there are at most $2D^{-5\gamma}n$ components in all. Thus, the forest $F'$ contains at least $(1-2D^{-5\gamma}) n - \frac{3}{4}D^{-\gamma} n \ge (1-D^{-\gamma}) n$ edges, as desired.
\end{proof}

\subsection{Proof of Theorem~\ref{thm:coloureddigraph}} \label{subsection:AHproof}

\begin{proof}[Proof of Theorem~\ref{thm:coloureddigraph}]
We begin by fixing parameters
\[
1/n_0  \ll \alpha \ll \gamma\ll \varepsilon \ll \mu.
\]
    Let $G$ be a copy of the complete digraph on $n \geq n_0$ vertices $\stackrel{\leftrightarrow}{K_n}$ (say, with a loop at each vertex), which is edge-coloured, using colours from a set $C$, so that each colour appears at most $n$ times. Further, we may assume that all but at most one of the colours appear at least $n/2$ times. We may view $G$ as a $3$-uniform, $3$-partite hypergraph instead. Indeed, let $H$ be the hypergraph on vertex classes $V_1,V_2$ and $C$, where $V_1$ and $V_2$ are both copies of $[n]$, the vertex set of $G$, to which we add the edge $uvc$ whenever $\chi({\stackrel{\longrightarrow}{uv}})=c$. It is easy to see that $G$ and $H$ are equivalent and an application of our main result, Lemma~\ref{lem:semiunifieddecomp}, gives a subhypergraph $\mathcal{H}$ of $H$, with an all-or-small codegree condition, which we view as the following (spanning) directed subdigraph $\mathcal{D}\subset G$ instead:

   \stepcounter{propcounter}
\begin{enumerate}[label = {{\textbf{\Alph{propcounter}\arabic{enumi}}}}]
        \item $\Delta(\mathcal{D}) \leq (1+n^{-\varepsilon})n^{1-3\mu}$,
        \item $e(\mathcal{D})\geq (1-5n^{-\varepsilon})n^{2-3\mu}$,
        \item each colour $c \in C$ appears at most $(1+n^{-\varepsilon})n^{1-3\mu}$ times,
        \item \label{prop:digraphallorsmall}for any vertex $v \in V(\mathcal{D)}$ and colour $c \in C$, either $c$ appears at most $n^{1-3\mu-\varepsilon}$ times among the in/(out)-edges of $v$, or each in/(out)-edge of $v$ has colour $c$, in which case $c$ does not appear elsewhere.
    \end{enumerate}

An application of the nibble at this point would give us a (rainbow) disjoint union of directed paths and cycles which are almost spanning. In order for us to obtain a single long (rainbow) cycle, we should first reserve sets of vertices and colours, apply the nibble to get a (rainbow) disjoint union of paths and cycles, clean things up and use these reserved vertices and colours to connect everything into a single long directed (rainbow) cycle. If we had a small codegree condition, it would be straightforward to reserve the vertices and colour uniformly at random. However, due to the all-or-small codegree condition that we do have, we should consider the relationship between vertices and colours. In particular, for example, if a vertex $v \in V(\mathcal{D})$ has high codegree with a colour $c$ among its out-edges say, then it would make sense that the colour $c$ should be reserved if and only if the vertex $v$ is reserved. Before continuing to construct these reservoirs, we note that almost all vertices in $V(\mathcal{D})$ have both large in-/out-degrees, as follows.

  \begin{claim}[Few bad vertices]\label{claim:fewbadvertices} There are at most $40n^{1-\varepsilon}$ vertices $v \in V(\mathcal{D})$ for which $d^\sigma(v) \leq n^{1-3\mu} /2$, for some $\sigma \in \{+,-\}$.
        
    \end{claim}
    \claimproofstart[Proof of Claim~\ref{claim:fewbadvertices}]
        In the corresponding $3$-uniform, $3$-partite hypergraph $H$ on parts $V_1,V_2$ and $C$ we know that $\sum_{x\in V_1}d(x)=e(H)\geq (1-5n^{-\varepsilon})n^{2-3\mu}$. Suppose then, that there are $k$ vertices in $V_1$ with degree $d_{H}(x)\leq n^{1-3\mu}/2$, then
        \[
        (1-5n^{-\varepsilon})n^{2-3\mu}\leq \sum_{x \in V_1}d_{H}(x) \leq k\cdot n^{1-3\mu}/2+(n-k)\cdot (1+n^{-\varepsilon})n^{1-3\mu},
        \]
        which gives $k \leq 20n^{1-\varepsilon}$. Since for $v \in V_1$, the edges incident to $v$ in $H$ correspond to the $+$-oriented edges of $v$ in $G$, and similarly, for $v' \in V_2$, the edges incident to $v'$ in $H$ correspond to the $-$-oriantated edges of $v'$ in $G$, it follows $\mathcal{D}$ has at most $40n^{1-\varepsilon}$ vertices with either $d^-(v)\leq n^{1-3\mu}/2$ or $d^+(v) \leq n^{1-3\mu}/2$, as desired.\claimproofend

    Further, let $\mathcal{B}$ denote the set of these bad vertices $v \in \mathcal{D}$ for which there is some $\sigma \in \{+,-\}$ for which $d^\sigma(v)\leq n^{1-3\mu}/2$ and so $|\mathcal{B}|\leq 40n^{1-\varepsilon}$.

    We call a triple $(v,\sigma,c)\in V(\mathcal{D}) \times \{+,-\}\times C$ a \textit{high codegree triple} if the colour $c$ appears on more than $n^{1-3\mu-\varepsilon}$ $\sigma$-oriented edges from $v$ in $\mathcal{D}$, in which case it follows by~\ref{prop:digraphallorsmall} that the colour $c$ appears on precisely all $\sigma$-oriented edges from $v$ in $\mathcal{D}$ and no other edges. For such a high codgeree triple $(v,\sigma,c)$ we define $hc(v,\sigma):=c$ to be the \textit{high codegree colour}, $hp(c):=(v,\sigma)
    $ to be the \textit{high codegree pair} and $hv(c):=v$ to be the \textit{high codegree vertex}. Note that each vertex $v \in V(\mathcal{D})$ could have up to two different high codegree colours, hence specifying the orientation $\sigma \in \{+,-\}$ is important. Let $C_h$ denote the set of all high codegree colours $c \in C$ and set $C_\ell =C \setminus C_h$ to be the set of \textit{low codegree colours}, i.e. all colours $c\in C$ which appear on at most $n^{1-3\mu-\varepsilon}$ $\sigma$-oriantated edges from $v$ in $\mathcal{D}$ for each $(v,\sigma)\in V(\mathcal{D})\times \{+,-\}$. Let $R=V(\mathcal{D})(2n^{-\alpha})$ be a random subset of vertices, each taken independently at random with probability equal to $2n^{-\alpha}$. We construct two sets of colour reservoirs, $C_1$ and $C_2$ (to be used to make connections within $R$ itself and between $R$ and its compliment, respectively) as follows. Let $C'_\ell=C_\ell (2n^{-\alpha})$ and $C_h'=\{c \in C_h:hv(c)\in R\}$ and note that both are random subsets of $C_\ell$ and $C_h$, respectively, where each element is taken independently at random with probability $2n^{-\alpha}$. Randomly partition further each of $C_\ell'$ and $C_h'$ into two sets by setting $C_{\ell,1}'=C_\ell'(1/2)$, $C_{\ell,2}'=C_{\ell}'\setminus C_{\ell,1 }'$ and $C_{h,1}'=C_h'(1/2)$, $C_{h,2}'=C_{h}'\setminus C_{h,1 }'$. Finally, set $C_1=C_{\ell,1}' \cup C_{h,1}'$ and $C_{2}=C_{\ell,2}'\cup C_{h,2}'$. Observe that each colour $c \in C$ is in $C_1$, $C_2$ or neither, with probabilities $n^{-\alpha}$, $n^{-\alpha}$ or $1-2n^{-\alpha}$ respectively, independently at random from all other colours.

 For a vertex $v \in V(\mathcal{D})$, subset of vertices $U \subset V(\mathcal{D})$, subset of colours $C'\subset C$ and orientation $\sigma \in \{+,-\}$, let $N_\mathcal{D}^\sigma(v,U,C')$ denote the $\sigma$-neighbourhood of $v$ in $\mathcal{D}$ inside $U$ whose edges are coloured with colours in $C'$. Further, for subsets $A,B \subset V(\mathcal{D})$, set of colours $C'\subset C$ and orientation $\sigma \in \{+,-\}$, let $e^{\sigma}_G(A,B,C')$ denote the number of $\sigma$-oriented edges from $A$ to $B$ whose colours are in $C'$, in the ambient digraph $G$. For each high codegree colour not in the colour reservoir, $c \in C_h \setminus C_h'$, with high codegree pair $hp(c)=(v,\sigma)$, let $s(c)$ be a random subset of the $\sigma$-oriented edges from $v$ in $\mathcal{D}$, each taken independently at random with probability $2n^{-\alpha}$. Further, let $S=\cup_{c \in C_h \setminus C_h'}s(c)$. So that $h(v,\sigma)$ is always defined, for convenience, whenever $(v,\sigma,c)$ is not a high codegree triple for any colour $c$, define $hc(v,\sigma):=\emptyset$.

\begin{claim} \label{claim:digraphrandomproperties}
    Let $\mathcal{D}' \subset \mathcal{D}$ be the digraph with vertex set $V(\mathcal{D}')=V(\mathcal{D})\setminus R$ and edge set $E(\mathcal{D}')= E(\mathcal{D}[V(\mathcal{D})\setminus R]) \setminus(S \cup \{e \in E(\mathcal{D}):\chi(e) \in C_1\cup C_2\})$. By a standard application of the union bound, together with the Chernoff and McDiarmid's inequalities we get the following.

   \stepcounter{propcounter}
\begin{enumerate}[label = {\textup{\textbf{\Alph{propcounter}\arabic{enumi}}}}]
        \item $|R|=2n^{1-\alpha} \pm n^{2/3}$,
        
        \item $|V(\mathcal{D}')|=(1-2n^{-\alpha})n\pm n^{2/3}$,
         \item $\Delta(\mathcal{D}')\leq (1-2n^{-\alpha})^2(1+n^{-\varepsilon})n^{1-3\mu }+n^{2/3}$, 
        \item $e(\mathcal{D}')\geq(1-2n^{-\alpha})^3(1-5n^{-\varepsilon})n^{2-3\mu}- n^{4/3}$,

\item for every vertex $v\in V(\mathcal{D}')\setminus \mathcal{B}$ and each orientation $\sigma \in \{+,-\}$ either 
        \[
        |N_\mathcal{D}^{\sigma}(v,R,C_2)| \geq \frac{n^{1-3\mu}}{2} \cdot 2n^{-2\alpha} -n^{2/3}\geq n^{1-3\mu-3\alpha},
        \]
        
        or, $(v,\sigma)$ is a high codegree pair and 
        
        \[
        |N_\mathcal{D}^{\sigma}(v,R,hc(v,\sigma))| \geq  \frac{n^{1-3\mu}}{2}\cdot 2n^{-\alpha}-n^{2/3}\geq n^{1-3\mu-3\alpha},
        \]\label{prop:edgesacross}

         \item for every pair of distinct vertices $a,b \in V(\mathcal{D})\setminus \mathcal{B}$, letting $A=N^+_\mathcal{D}(a,R,C_2 \cup hc(a,+))$ and $B=N^-_\mathcal{D}(b,R,C_2\cup hc(b,-))$, we have that 
         
        \[
        e^{+}_G(A,B,C_1)\geq n^{2-6\mu-8\alpha}.
        \] \label{prop:edgeswithin}

    \end{enumerate}
   
\end{claim}
  \claimproofstart[Proof of Claim~\ref{claim:digraphrandomproperties}]
      The first four conditions fail with probability exponentially small in $n$, by Chernoff's inequality.
      \noindent \ref{prop:edgesacross}: Fix a vertex $v \in V(\mathcal{D}')\setminus \mathcal{B}$ and an orientation $\sigma \in \{+,-\}$. If $(v,\sigma)$ is a high codegree pair then $|N^\sigma_\mathcal{D}(v,R,h(v,\sigma)|$ is a binomial random variable with at least $n^{1-3\mu}/2$ trials, as $v \notin \mathcal{B}$, and success probability $2n^{-\alpha}$. Hence, $\mathbb{E}|N^\sigma_\mathcal{D}(v,R,h(v,\sigma)| \geq \frac{n^{1-3\mu}}{2}\cdot 2n^{-\alpha}$ and so $|N^\sigma_\mathcal{D}(v,R,h(v,\sigma)| \geq \frac{n^{1-3\mu}}{2}\cdot 2n^{-\alpha} -n^{2/3}$ with probability $1-\exp(-n^{\Omega(1)})$, by Chernoff's inequality. In the former case, if $(v,\sigma)$ is not a high codegree pair then $|N^\sigma_\mathcal{D}(v,R,C_2)|$ stochastically dominates\footnote{We say that a random variable $X$ stochastically dominates a random variable $Y$ if for all real $t$, $\mathbb{P}(X\geq t)\geq \mathbb{P}(Y\geq t)$.} a binomial random variable, with at least $n^{1-3\mu}/2$ trials and with success probability at least $2n^{-2\alpha}$. Indeed, a neighbour $u$ of $v$ in $\mathcal{D}$ is in $R$ with probability $2n^{-\alpha}$ and for such a vertex $u \in N^\sigma_\mathcal{D}(v)\cap R$, the colour of that edge is in $C_2$ with probability either $n^{-\alpha}$ or $1/2$, if the colour is a low codegree colour or is a high codegree colour of $u$, respectively. It follows from Chernoff once again that, $|N^\sigma_\mathcal{D}(v,R,C_2)|\geq \frac{n^{1-3\mu}}{2}\cdot 2n^{-2\alpha}-n^{2/3}$ with probability $1-\exp(-n^{\Omega(1)})$. Therefore, with high probability, this holds true for every vertex $v \in V(\mathcal{D'})\setminus\mathcal{B}$ and orientation $\sigma \in \{+,-\}$. 
      
      \noindent \ref{prop:edgeswithin}: Fix such a pair $a,b \in V(\mathcal{D})\setminus \mathcal{B}$. Set $A=N_\mathcal{D}^+(a,R,C_2\cup hc(a,+))$ and $B=N_\mathcal{D}^-(b,R,C_2\cup hc(b,-))$. For each vertex and low codegree colour $j^\star \in V(\mathcal{D})\cup C_\ell$, let $X_{j^\star}$ be the random variable which indicates in which of $R$, $V(\mathcal{D'})$ or $C_1$, $C_2$, $C\setminus (C_1 \cup C_2)$, respectively, $j^\star$ belongs to. Let $X=(X_1,X_2\dots,X_k)$ be the collection of all of these $X_{j^\star}$'s and let $f(X)=e_G^+(A,B,C_1)$, which is fully determinable by the vector of random variables $X$. Observe that there are at most $k\leq 4n$ $X_{j^\star}$'s, as there are at most $n$ vertices and $3n$ colours, and changing the value of some $X_{j^\star}$ changes $f(X)$ by at most $c_{j^\star}\leq 2n$, as each vertex is incident to at most $2n$ edges and a colour appears at most $n$ times. Note that $f(X)$ counts the number of directed paths with length three between $a$ and $b$; with internal vertices in $R$; first edge having colour in $hc(a,+)$ if $(a,+,hc(a,+))$ is a high codegree triple or colour in $C_2$ otherwise; second edge having colour in $C_1$; last edge having colour in $hc(b,-)$, if $(b,-,hc(b,-))$ is a high codegree triple or colour in $C_2$ otherwise. By linearity of expectation, $\mathbb{E}f(X)$ is equal to the sum over all viable paths of the probability that the path satisfies the above conditions. Since $a,b \notin \mathcal{B}$, there are at least $(n^{1-3\mu}/2-1)^2-n^{1-3\mu}/2 \geq n^{2-6\mu}/10$ possible such paths. For each such path, the probability it satisfies the above conditions is at least $(n^{-\alpha})\cdot(2n^{-\alpha})\cdot (n^{-\alpha})\cdot (2n^{-\alpha})\cdot (n^{-\alpha})=4n^{-5\alpha}$. Hence we see that 
      \[
      \mathbb{E}f(X)\geq \frac{n^{2-6\mu}}{10}\cdot4n^{-5\alpha} \geq n^{2-6\mu-7\alpha}.
      \]

      Since $\sum c_i^2 \leq 10n^3$, applying Lemma~\ref{lem:mcdiarmid} with $t= n^{2-6\mu-8\alpha}/2$ gives 
      \[
      \mathbb{P}(|f(X)-\mathbb{E}f(X)|>n^{2-6\mu -8\alpha}/2)\leq 2\exp\left(-\frac{n^{4-12\mu -16\alpha}}{40n^3}\right)=\exp(-n^{\Omega(1)}).
      \]
      Hence, $f(X)\geq n^{2-6\mu-8\alpha}$ with probability $1-\exp(-n^{\Omega(1)})$. Taking a union bound over at most $n^2$ pairs of vertices concludes the proof.\claimproofend

\begin{claim}\label{claim:rainbowpathforest}
   There exists a constant $\gamma \ll \varepsilon$ so that the following holds. There exists a collection of rainbow directed paths, $\mathcal{P}=\{P_1,P_2,\dots,P_L\}$ in $\mathcal{D}'$, with endpoints $U=\{u_1,u_2\dots,u_L\}$ and $V=\{v_1,v_2,\dots,v_L\}$, where $P_i$ has endpoints $u_i,v_i$ with $d^+_{P_i}(u_i)=d^-_{P_i}(v_i)=1$, satisfying the following properties.
    \stepcounter{propcounter}
\begin{enumerate}[label = {\textup{\textbf{\Alph{propcounter}\arabic{enumi}}}}]
        \item $\mathcal{P}$ contains no bad vertices, $(\cup_{P \in \mathcal{P}}V(P))\cap \mathcal{B}=\emptyset$.\label{prop:linearpathforest:nobadvtx}

    \item For each $r \in R$ and $\sigma \in \{+,-\}$, $|N^\sigma_\mathcal{D}(r)\cap (U\cup V)|\leq n^{1-3\mu -\gamma +2\alpha}$, and for each $c_2 \in C_2$ there are at most $n^{1-3\mu -\gamma +2\alpha}$ edges with colour $c_2$ incident to vertices in $U \cup V$ in $\mathcal{D}$.\label{prop:linearpathforest:randomendpoints}

    \item For all $r_1,r_2\in R$, $(v_i,u_{i+1})\in N_{\mathcal{D}}^-(r_1)\times N_\mathcal{D}^+(r_2)$, for at most $n^{1-6\mu -\gamma/2 +5\alpha}$ indices $i \in [L]$.\label{prop:linearpathforest:randompaths}
  \item Each $P \in \mathcal{P}$ has size $|P| \in [n^{\gamma}/2,2n^{\gamma}]$.\label{prop:linearpathforest:pathlength}
      \item  $\mathcal{P}$ is almost spanning, $|\cup_{P \in \mathcal{P}}V(P)|\geq (1-5n^{-\alpha})n$.\label{prop:linearpathforest:spanning}
  
    \end{enumerate}
\end{claim}
\claimproofstart[Proof of Claim~\ref{claim:rainbowpathforest}]
    By adding isolated vertices to $\mathcal{D}'$, we may assume that it is a digraph on $m:=(1-2n^{-\alpha})(1+n^{-\varepsilon})n$ vertices. Furthermore, from the above properties, $\mathcal{D}'$ satisfies:
    \begin{enumerate}
        \item $\Delta(\mathcal{D}')\leq D:=(1-2n^{-\alpha})^2(1+n^{-\varepsilon})^2n^{1-3\mu}$,
        \item $e(\mathcal{D}')\geq (1-2n^{-\alpha})^3(1-6n^{-\varepsilon})n^{2-3\mu}\geq \frac{(1-6n^{-\varepsilon})}{(1+n^{-\varepsilon})^3}mD\geq (1-D^{-\varepsilon/2})mD$,
        \item for every vertex $v \in V(\mathcal{D}')$, colour $c \in C$ and orientation $\sigma \in \{+,-\}$, either $\text{cod}_c^{\sigma}(v) \leq n^{1-3\mu-\varepsilon}\leq D^{1-\varepsilon/2} $, or all $\sigma$-oriented edges from $v$ are coloured with $c$ and the colour $c$ does not appear elsewhere.
    \end{enumerate}
    An application of Theorem~\ref{lem:noshortcycles} with $\delta_{\ref{lem:noshortcycles}}=\varepsilon/2$ gives a $\gamma_{\ref{lem:noshortcycles}}$ for which there is a rainbow directed path forest $\mathcal{P}^\star$ with $(1)$ $|\cup_{P \in \mathcal{P}^\star}V(P)|=(1-D^{-\gamma_{\ref{lem:noshortcycles}}})m$ and $(2)$ each $P \in \mathcal{P}^\star$ has size $|P| \in [D^{\gamma_{\ref{lem:noshortcycles}}} ,2D^{\gamma_{\ref{lem:noshortcycles}}}]$. By observing that $D^{\gamma_{\ref{lem:noshortcycles}}}=n^{\gamma}$ for some $\gamma_{\ref{lem:noshortcycles}}/2\leq \gamma \leq \gamma_{\ref{lem:noshortcycles}}$, we may instead write that $(1)$ $|\cup_{P\in P^\star}V(P)|=(1-n^\gamma)|V(\mathcal{D}')|$ and $(2)$ each $P \in \mathcal{P}^\star$ has size $|P|\in [n^\gamma,2n^\gamma]$.

    \noindent\ref{prop:linearpathforest:nobadvtx}: We remove entirely from $\mathcal{P}^\star$ any paths containing bad vertices from $\mathcal{B}$. This removes at most $40n^{1-\varepsilon}$ paths $P \in \mathcal{P}^\star$ and at most $40n^{1-\varepsilon}\cdot 2n^\gamma \leq n^{1-\gamma}$ vertices. 
    
    \noindent\ref{prop:linearpathforest:randomendpoints}: For each remaining path $P \in \mathcal{P}^\star$, uniformly at random delete an edge among the first $n^{\gamma-\alpha}$ edges and among the last $n^{\gamma -\alpha}$ edges, disconnecting each path into three segments. Let $\mathcal{P}'$ denote the collection of rainbow directed paths obtained by retaining only the middle path segment, for each $P \in \mathcal{P}^\star$. Cruicially, the probability that any vertex is an endpoint of a path $P \in \mathcal{P}'$ is at most $n^{-\gamma +\alpha}$. In particular, $\mathcal{P}'$ can be chosen so that every vertex $r \in R$ has at most $n^{1-3\mu -\gamma +2\alpha}$ edges incident to endpoints  in $\mathcal{D}$ and further, so that every colour $c_2\in C_2$ has at most $n^{1-3\mu-\gamma+2\alpha}$ edges of its colour incident to endpoints, in $\mathcal{D}$. This can be shown, for example, by using the union bound together with Chernoff's inequality for vertices $r \in R$ and McDiarmid's inequality for colours $c_2 \in C_2$, since a vertex in $V(\mathcal{D}')$ and a colour in $C_2$ share at most $n^{1-3\mu-\varepsilon}$ edges and $\gamma \ll \eps$. This deletes at most $2n^{\gamma-\alpha}\cdot n/n^\gamma=2n^{1-\alpha}$ vertices.

    \noindent\ref{prop:linearpathforest:randompaths}:  The number of paths present satisfies
\[
n^{1-\gamma} \geq |\mathcal{P}'| \geq \frac{(1-n^{\gamma})((1-2n^{-\alpha})n-n^{2/3})}{2n^\gamma}-40n^{1-\varepsilon}\geq n^{1-\gamma}/4.
\]
So let $t:=|\mathcal{P}'|$. Let $L:=t-n^{1-3\gamma/2}$ and select uniformly at random, without replacement, and one at a time, $P_1,P_2\dots, P_L\in \mathcal{P}'$. Denote the endpoints of $P_i$ by $u_i$, and $v_i$, so that $u_i$ has out-degree one in $P_i$ and $v_i$ has in-degree one in $P_i$, say, and set $U=\{u_1,u_2,\dots,u_L\}$ and $V=\{v_1,v_2,\dots,v_L\}$ to be the sets of endpoints. We will connect $v_i$ to $u_{i+1}$ (modulo $L$) with rainbow paths of length three via $R,C_1,C_2$. Due to our appropriate choice of $L$, the probability that an endpoint $x$ was matched to any other endpoint $y$ is at most $n^{-1+3\gamma/2}$. Now let $r_1,r_2 \in R$ be distinct vertices. By~\ref{prop:linearpathforest:randomendpoints} above we know that 
\[
|N_\mathcal{D}^-(r_1) \cap V|\leq n^{1-3\mu -\gamma +2\alpha}~~\text{and}~~|N_\mathcal{D}^+(r_2)\cap U|\leq n^{1-3\mu-\gamma+2\alpha}.
\]
Then, the expected number of indices $i \in [L]$ for which $(v_i,u_{i+1})\in N_\mathcal{D}^-(r_1)\times N_\mathcal{D}^+(r_2)$ is at most
\[
\sum_{i\in [L]:v_i \in N_\mathcal{D}^-(r_1)} \mathbb{P}[u_{i+1}\in N_\mathcal{D}^+(r_2)]\leq n^{1-3\mu-\gamma +2\alpha}\cdot \frac{n^{1-3\mu-\gamma +2\alpha}}{n^{1-3\gamma/2}}=n^{1-6\mu-\gamma/2+4\alpha}.
\]
Applying the Chernoff bound and taking a union bound over all pairs $r_1,r_2 \in R$, guarantees that for every pair of vertices $r_1,r_2 \in R$, $(v_i,u_{i+1})\in N_{\mathcal{D}}^-(r_1)\times N_\mathcal{D}^+(r_2)$, for at most $n^{1-6\mu -\gamma/2 +5\alpha}$ indices $i \in [L]$.

\noindent\ref{prop:linearpathforest:pathlength}: Let $\mathcal{P}=\{P_1,P_2,\dots,P_L\}$ denote the desired collection of rainbow directed paths. Each path lost at most $2n^{\gamma-\alpha }$ vertices and no path gained vertices. Thus, each $P \in \mathcal{P}$ has size $|P| \in [n^{\gamma}-2n^{\gamma-\alpha},2n^\gamma]\subseteq [n^\gamma/2,2n^\gamma]$.

\noindent\ref{prop:linearpathforest:spanning}: The number of vertices in $V(\mathcal{D})\setminus(\cup_{P\in \mathcal{P}}|V(P)|)$ is at most
\[
(2n^{1-\alpha}+n^{2/3})+n^{1-\gamma}+n^{1-\gamma}+2n^{1-\alpha}+n^{1-3\gamma/2}\cdot 2n^\gamma\leq 5n^{1-\alpha}.
\]

Where each term respectively arises due to: the size of $|R|$, leftover given by Theorem~\ref{lem:noshortcycles}, deleting paths containing vertices from $\mathcal{B}$, shortening the paths and finally deleting paths when randomly ordering $P_1,P_2,\dots,P_L$.\claimproofend

\textbf{Auxiliary Hypergraphs.} Before we can apply Theorem~\ref{thm:aharonihaxell}, we need to re-colour $\sigma$-oriented edges from $v \in V(\mathcal{D}')$ which only have colour $hc(v,\sigma)$ appearing on them. Indeed, for each $c \in C_h \setminus C_h'$ with $hp(c)=(v,\sigma)$ recolour all $\sigma$-oriented edges from $v$ in $\mathcal{D}$ with new distinct colours. Let $Z$ denote the colleciton of all of these new distinct colours across all $c \in C_h \setminus C_h'$. Note that any rainbow directed path/cycle we find under this recolouring, in the manner in which we are going to find it, will persist to be a rainbow path/cycle in the original digraph
$G$. Indeed, since $c \in C_h \setminus C_h'$, $c \notin C_1 \cup C_2$ so will not be used to make connections and in $\mathcal{D}$, $c$ appeared on all $\sigma$-oriented edges from $v$, for $hp(c)=(v,\sigma)$, and on no other edges, so any directed path/cycle will use it at most once.

For $i \in [L]$, define $H_i$ to be the $5$-uniform (multi)-hypergraph on vertex set $V(H_i)=R \cup C_1 \cup C_2 \cup Z$ to which we add all edges that correspond to a rainbow $v_i$-$u_{i+1}$ path with length $3$, internal vertices in $R$ and using colours in $C_2 \cup Z$ for the first and last edge and colour in $C_1$ for the middle edge. In particular, we let
\[
E(H_i)=\{axc_1yb:a,b\in C_2\cup Z,~c_1\in C_1~,x,y\in R,~a\neq b,~x\neq y,
\chi(\stackrel{\longrightarrow}{v_ix})=a,~\chi(\stackrel{\longrightarrow}{xy})=c_1,~ \chi(\stackrel{\longrightarrow}{yu_{i+1}})=b
\}.
\]

Note that each $H_i$ is technically a multihypergraph, but any edge has multiplicity at most two (whenever $\chi(\stackrel{\longrightarrow}{v_iy})=a$ and $\chi(\stackrel{\longrightarrow}{xu_{i+1}})=b$). For $B \subset [L]$, we let $H(B):= \cup_{i\in B}H_i$. We prove the following properties about $H(B)$.

\begin{claim}\label{claim:aharonihaxell}
    Let $B \subset[L]$ be a subset with size $|B|=k$. The multihypergraph $H(B)=\cup_{i \in B}H_i$ satisfies
    \stepcounter{propcounter}
\begin{enumerate}[label = {\textup{\textbf{\Alph{propcounter}\arabic{enumi}}}}]

        \item $e(H(B))\geq k\cdot n^{2-6\mu-8\alpha}$, \label{prop:aharonihaxelledges}
        \item $\Delta(H(B))\leq n^{2-6\mu -\gamma/4}$. \label{prop:aharonihaxellmaxdegree}
    \end{enumerate}
\end{claim}
\claimproofstart[Proof of Claim~\ref{claim:aharonihaxell}]
\noindent\ref{prop:aharonihaxelledges}: It suffices to show that for any $i\in [L]$, $e(H_i)\geq n^{2-6\mu-8\alpha}$. This is an immediate consequence of~\ref{prop:edgeswithin} applied with $a=v_i$ and $b=u_{i+1}$. 

\noindent\ref{prop:aharonihaxellmaxdegree}: It suffices to show that $\Delta(H([L]))\leq n^{2-6\mu -\gamma /4}$. \noindent \textbf{Case 1.} If $x \in Z$ then $d_{H([L])}(x)\leq 1$, as each colour in $Z$ is unique. \noindent \textbf{Case 2.} If $x \in R\cup C_2$ then by~\ref{prop:linearpathforest:randomendpoints} at most $n^{1-3\mu-\gamma+2\alpha}$ edges incident to $x$/edges coloured with $x$ (depending if $x \in R$ or $x \in C_2$) are incident to endpoints in $U \cup V$. For any such endpoint, say $v_i \in V$, $d_\mathcal{D}(u_{i+1})\leq (1+n^{-\eps})n^{1-3\mu}$ and so, $x$ is contained in at most $n^{1-3\mu-\gamma+2\alpha} \cdot (1+n^{-\eps})n^{1-3\mu}\leq n^{2-6\mu -\gamma/4}$ edges in $H([L])$, as desired. \noindent \textbf{Case 3.} Thus we may assume $x \in C_1$. We know by~\ref{prop:linearpathforest:randompaths} that for any pair of vertices $r_1,r_2$ that $(v_i,u_{i+1})\in N_{\mathcal{D}}^-(r_1)\times N_\mathcal{D}^+(r_2)$, for at most $n^{1-6\mu -\gamma/2 +5\alpha}$ indices $i \in [L]$. Since there are at most $n$ edges in $R$ coloured by $x \in C_2$ in $G$, it follows that $d_{H([L])}(x)\leq n^{2-6\mu -\gamma/2+5\alpha }\leq n^{2-6\mu -\gamma/4 }$, concluding the proof.\claimproofend

We can now then find, in $H(B)$, a matching with size at least $kn^{\gamma/4 -8\alpha}/5\geq 5k$, by repeatedly taking out an edge one at a time, as doing so will delete at most $5\Delta(H(B))$ edges at each step. It then follows from Theorem~\ref{thm:aharonihaxell} that $H([L])$ has a system of disjoint representatives, which precisely means that we can connect $\{P_1,P_2,\dots,P_L\}$ into a single rainbow directed cycle with size at least $(1-5n^{-\alpha})n$.
\end{proof}

\end{document}